\documentclass[12pt, twoside]{article}
\usepackage{amsmath,amsthm,amssymb}
\usepackage{times}
\usepackage{enumerate}
\usepackage{graphicx}
\usepackage[utf8]{inputenc}
\usepackage{nicefrac}
\usepackage{hyperref}
\usepackage{amsrefs}
\usepackage{nowidow}

\allowdisplaybreaks

\newtheorem{theorem}{Theorem}[section]
\newtheorem{lemma}[theorem]{Lemma}

\theoremstyle{definition}
\newtheorem{definition}[theorem]{Definition}

\newtheorem{prop}[theorem]{Proposition}
\newtheorem{cor}[theorem]{Corollary}
\newtheorem{ass}{Assumption}

\theoremstyle{remark}
\newtheorem{remark}[theorem]{Remark}

\numberwithin{equation}{section}

\theoremstyle{definition}

\numberwithin{equation}{section}

\newcommand{\subjclass}[1]{\bigskip\noindent\emph{2010 Mathematics Subject Classification:}\enspace#1}
\newcommand{\keywords}[1]{\noindent\emph{Keywords:}\enspace#1}

\begin{document}


\baselineskip=17pt


\title{An Obstacle Problem for Elastic Curves: \\ Existence Results}

\author{Marius Müller\\
Universität Ulm, Helmholtzstraße 18, 89081 Ulm, Germany\\
marius.mueller@uni-ulm.de\\ }
\date{\today}

\maketitle


\begin{abstract}
We consider an obstacle problem for elastic curves with fixed ends.  We attempt to extend the graph approach provided in \cite{Anna}. More precisely, we investigate nonexistence of graph solutions for special obstacles and extend the class of admissible curves in a way that an existence result can be obtained by a penalization argument.


\end{abstract}

\section{Introduction}
\subsection{Model and main results} 
The energy considered in this article is the Euler-Bernoulli energy or simply elastic energy, given by
\begin{equation*}
\mathcal{E}(\gamma) = \int_\gamma \kappa^2(s) d\mathbf{s},
\end{equation*}
where $\kappa$ denotes the curvature of a sufficiently smooth planar curve $\gamma$  and $d\mathbf{s}$ denotes the arclength parameter. In what follows, we will fix the endpoints of $\gamma$, so we can assume that $\gamma: [0,1] \rightarrow \mathbb{R}$ is such that $\gamma(0) =(0,0)^T$ and $\gamma(1) = (1,0)^T$. The elastic energy is well-defined on curves with at least two weak derivatives, which are additionally immersed, i.e. there is an $\epsilon > 0 $ satisfying $|\gamma'| > \epsilon$ on $ (0,1)$. In this case  $\mathcal{E}(\gamma)$ can be rewritten as 
\begin{equation*}
\mathcal{E}(\gamma ) = \int_0^1 \frac{\langle \gamma'', N_\gamma \rangle^2}{|\gamma'|^3} dt,
\end{equation*}
where
\begin{equation*}
N_\gamma := \frac{1}{|\gamma'|} \begin{pmatrix}
-\gamma_2' \\ \gamma_1' 
\end{pmatrix}
\end{equation*}
is the unit normal vector associated to the tangential vector $\gamma'$ in $\mathbb{R}^2$. 

We are interested in an obstacle problem, so we have to impose that the curve $\gamma$ lies above a given obstacle, which we will usually call $\psi$. 
\begin{ass}(Assumptions on the obstacle) \label{ass:obst}

In what follows, the obstacle $\psi \in C^0([0,1]) $ shall satisfy the following conditions:
\begin{equation*}
\psi(0),\psi(1) < 0, \quad 
\max_{x \in [0,1]} \psi(x) > 0. 
\end{equation*}
Moreover, at each point $x \in (0,1)$ there exist $\partial_+ \psi, \partial_-\psi$, the left and right sided first derivatives of $\psi$ and there is $C > 0 $ such that $|\partial_+ \psi (x) |, |\partial_- \psi(x) | \leq C$ for each $x \in (0,1)$. The smallest such $C$ will be denoted by $||\psi'||_\infty$, with an abuse of notation. 
\end{ass}

The obstacle condition can be understood as a confinement. The problem of minimizing the elastic energy subject to confinements has recently raised some interest, for example in \cite{Dayrens}, minimizing the same energy on closed curves confined to a bounded domain. The minimization of higher order functionals with obstacle constraints is a vivid field of research, with important contributions to be found in \cite{Caffarelli} for the biharmonic operator including results on regularity and the behavior of the free boundary. Moreover, \cite{Okabe} introduces a steepest energy descent flow for the energy associated to the biharmonic operator that respects the obstacle constraint.

In \cite{Anna}, the same obstacle problem as the given one is investigated, with the additional assumption that $\gamma$ is a graph, i.e $\gamma$ possesses a reparametrization of the form $(\cdot, u(\cdot))$ for some $u \in W^{2,2}(0,1)$. The elastic energy is then given by
\begin{equation}\label{eq:graphenergy}
\mathcal{E}(u) := \mathcal{E}(\gamma) = \int_0^1 \frac{u''(x)^2}{(1 + u'(x)^2)^\frac{5}{2}} dx, \quad \textrm{since} \; \; \; \kappa(x) = \frac{u''(x)}{(1+ u'(x)^2)^\frac{3}{2}}.
\end{equation}


The assumption in \cite{Anna} that $u \in W^{2,2}(0,1)$ tacitly imposes the condition that $u' \in L^\infty(0,1)$, which will turn out to be restrictive. Under this additional restriction, \cite{Anna} was able to show existence of a solution provided that $\psi$ satisfies certain smallness conditions. Numerics suggested that there might not be a graph solution in case that these conditions are violated, though. Indeed, \cite{Anna} contains a nonexistence result when minimizing in the class of symmetric graphs. We will improve this result and get rid of the symmetry assumption for certain obstacles. Therefore, for the rest of this article the admissible set of functions will be 
\begin{equation}\label{eq:graph set}
G_\psi := \{ u \in W^{2,2}(0,1) | \; u(0)= u(1) = 0 , u(x) \geq \psi(x) \; \forall x \in [0,1] \} . 
\end{equation}
We are now able to state the main nonexistence result of this article.
\begin{theorem}\label{thm:nonexcone} (Nonexistence of graph solutions for large cone obstacles) 

Let $\psi$ be a symmetric cone obstacle (see Definition \ref{def:symcone}).
Then $\inf_{ u \in G_\psi} \mathcal{E}(u)$ is not attained, if 
\begin{equation*}
\sup_{x \in (0,1)} \psi(x) > \sup_{A> 0 } \frac{1}{3}A \frac{\mathit{HYP2F1}( 1, \frac{3}{2} ; \frac{7}{4} , -A^2) }{\mathit{HYP2F1}(\frac{1}{2},1, \frac{3}{4}, -A^2)} \quad  [ \simeq 0.834626 ],  
\end{equation*}
where $\mathit{HYP2F1}$ denotes the hypergeometric function, see Definition \ref{def:hypgeo}. 
\end{theorem}
\begin{figure}[ht]
\centering 
\includegraphics[scale=0.63]{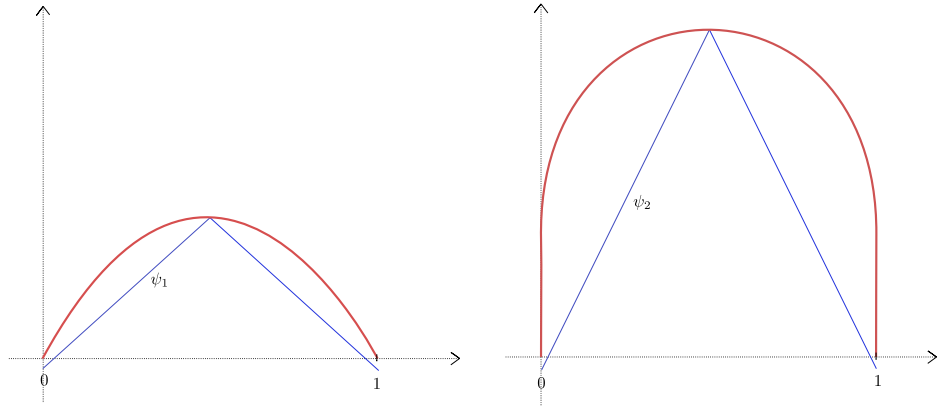}
\caption{A small cone obstacle where a minimizer in $G_\psi$ exists and a large one, where it does not - at least not as a graph.}\label{fig:1}
\end{figure}

 The value $\simeq 0.83 $ found numerically in Theorem \ref{thm:nonexcone} matches up with the result in \cite[Lemma 4.2]{Anna}, obtained for symmetric graphs. In this case - using important results from \cite{Deckelnick} - a function $U_0$ (see \eqref{def:Unull}) is found such that $\psi < U_0$ implies existence of a solution in $G_\psi$ and if $\psi(x_0) > U_0(x_0)$ for some $x_0 \in (0,1)$ it seems unlikely that a solution is to be found in the class of graphs at all. Note that
\begin{equation*}
 \max_{x\in [0,1]} U_0(x) = \frac{1}{\int_0^\infty \frac{1}{(1+ s^2)^\frac{5}{4}} ds } \simeq 0.834626,
 \end{equation*} 
 so indeed we found analytical evidence for the numerical evidence given in \cite[p.18]{Anna}.

The aforementioned theorem gives rise to the question whether a solution can always be found in a larger set. It is worth noting that a too liberal framework would lead to nonexistence of minimizers again: Define
\begin{eqnarray*}
K_\psi & := &  \{ \gamma \in W^{1,1}((0,1); \mathbb{R}^2) \cap W^{2,1}_{loc}((0,1); \mathbb{R}^2) | \;  \gamma(0) = (0,0)^T, \gamma(1) =(0,1)^T, \nonumber \\  &&  \qquad \qquad  \qquad \qquad \qquad \qquad \; \; \;   \gamma_2(t) \geq \overline{\psi}(\gamma_1(t)) \: \forall t \in [0,1] \},
\end{eqnarray*}
where $\overline{\psi}$ denotes the continuous extension of $\psi$ to $\mathbb{R}$ by a constant function. Then, 
\begin{equation*}
\inf_{\gamma \in K_\psi} \mathcal{E}(\gamma) = 0,
\end{equation*}
which is not attained, since only straight lines can have vanishing curvature. Figure \ref{fig:2} shows how to construct an element of $K_\psi$ with arbitrarily small energy.

\begin{figure}[ht]
\centering
\includegraphics[scale=0.55]{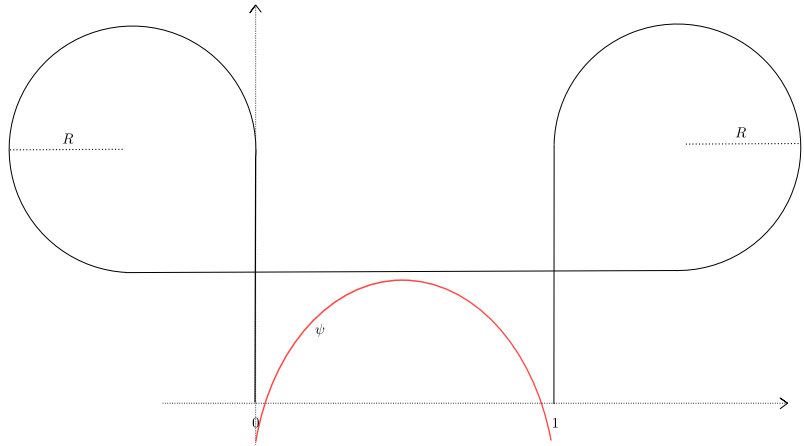}
\caption{A curve consisting of two three-quarter-circles of Radius $R$ connected by lines that lift it over $\psi$.}\label{fig:2}
\end{figure}

  The authors in \cite{Anna} suggest that the reason for nonexistence for large obstacles is blow-up of the derivatives near the boundary. So, the new framework should allow admissible curves to have vertical tangent lines.

For this, we will introduce the notion of pseudographs. 

\newpage

\begin{definition}(Pseudograph) \label{prop:p10}

A curve $\gamma \in W^{1,1}((0,1);\mathbb{R}^2)  \cap W^{2,1}_{loc}((0,1);\mathbb{R}^2)$ is called \emph{pseudograph} if it is immersed and $\gamma_1'(t) \geq 0 $ for each $t \in (0,1)$ and whenever $\gamma_1'(t) > 0$ on some interval $(a,b)$, then  $\gamma_2 \circ \gamma_1^{-1} \in W^{2,2}_{loc}(\gamma_1(a),\gamma_1(b))$. 
\end{definition}

In particular, such curves may have vertical tangent lines. The set of pseudographs can be thought of as a closure of the graphs in the topology of $W^{1,1} ( (0,1) ; \mathbb{R}^2) \cap W^{2,1}_{loc}((0,1), \mathbb{R}^2) $. The reason why we can only require $L^1$-integrability of the derivatives is that other spaces are not closed under the reparametrizations we use.

The pseudograph approach is a little delicate because of the following: If a curve $\gamma = \gamma(t)$ is indeed vertical at the boundary, then we can get a new curve prolonging the vertical parts. The curve remains in the class of pseudographs, above the obstacle, and the elastic energy remains unchanged (see Figure \ref{fig:3}).

\begin{figure}[ht]
\centering
\resizebox{\height}{10cm}{\includegraphics[scale=0.43]{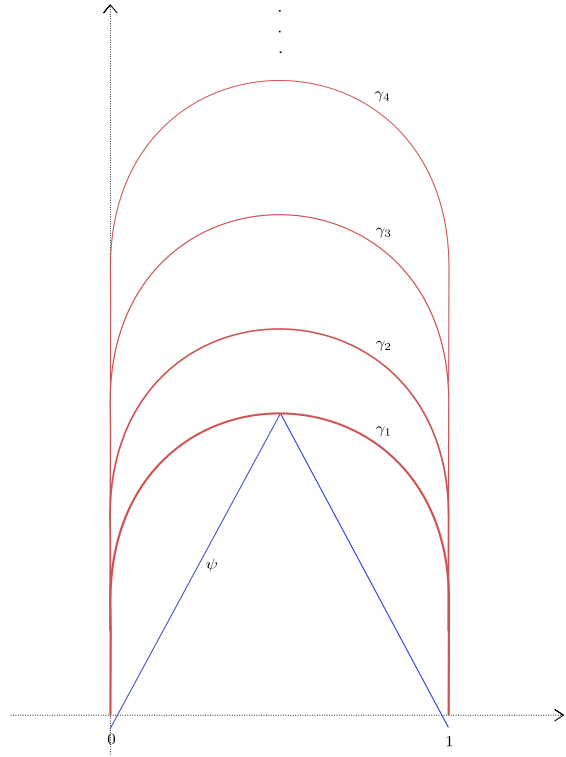}} 
\caption{A (possibly) minimizing sequence without a converging subsequence.}\label{fig:3}
 \end{figure}

Therefore, bounded length of an arbitrary minimizing sequence cannot be expected and so direct methods do not apply in an obvious way. To obtain an existence result, we have to pick a minimizing sequence that does not "beat loose" in the sense that it can develop arbitrarily large vertical parts. In order to do so, we will use a penalization technique. This step was not needed for the study of bounded confinements in \cite{Dayrens}, but this article points out that unbounded confinements are interesting showing that any bounded confinement is touched. 

The penalization term is going to be the length, which was chosen mainly because it is invariant with respect to reparametrization and also because of previous profitable examinations in \cite{Miura2}. Hence, we consider for any $\epsilon \geq 0 $ the penalized energy 
\begin{equation}\label{eq:pen}
\mathcal{E}_\epsilon(\gamma) := \int_\gamma \kappa^2 ds + \epsilon \int_\gamma ds  = \mathcal{E}(\gamma) + \epsilon \mathcal{L}(\gamma).
 \end{equation}
Note that $\mathcal{E}_0 = \mathcal{E}$. It will not be surprising that for each $\epsilon > 0 $, $\mathcal{E}_\epsilon$ admits a minimizer in
\begin{equation} \label{eq:Pseudo}
P_\psi := \{ \gamma  \; \textrm{pseudograph} : \gamma(0) = (0,0)^T, \gamma(1)= (1,0)^T, \gamma_2 \geq \psi \circ \gamma_1  \} . 
\end{equation} 
We attempt to understand the shape of this minimizer and see how the problem evolves as $\epsilon$ becomes small.

We examine properties of functions with small energy using geometric measure theory to obtain a concavity result inspired by \cite[Lemma 2.1]{Anna}. The main result is going to be the following:

\begin{theorem}\label{thm:panettone} (The energy-efficient $\cap$-shape)

Let $\gamma_0 \in P_\psi$. Consider
\begin{align}\label{eq:panettone}
B_{\gamma_0} := &\big\{ \gamma \in W^{1,1}((0,1);\mathbb{R}^2) \cap W^{2,1}_{loc}((0,1);\mathbb{R}^2)  :  \textrm{$\gamma$ is an immersed pseudograph and} \\*  & \qquad \exists 0 \leq \beta_1 < \beta_2 \leq 1 : (\gamma_1 )_{\mid_{[0, \beta_1]}} = 0 , \\ & \qquad (\gamma_1 )_{\mid_{[\beta_2 ,1]}} = 1, \textrm{while} \; \gamma_1'(t) > 0  \;  \forall t \in (\beta_1, \beta_2)  \;  \textrm{and if $u$ denotes the graph} \\* & \qquad \textrm{reparametrization of $\gamma_{\mid_{(\beta_1, \beta_2)}} $ then }  \textrm{$u$  satisfies $u \circ \gamma_{0,1} \geq \gamma_{0,2} $} \big\} \subset P_\psi  
\end{align}
the so-called set of all \textit{$\cap$-shaped pseudographs} lying above $\gamma_0$. Then there is $\gamma \in B_{\gamma_0} $ such that the graph reparametrization $u$ of $\gamma$ is concave and $\mathcal{E}(\gamma) \leq \mathcal{E}(\gamma_0)$ as well as $\mathcal{L}(\gamma) \leq \mathcal{L}(\gamma_0)$. In particular, for each $\epsilon \geq 0 $,
\begin{equation*}
\mathcal{E}_\epsilon(\gamma) \leq \mathcal{E}_\epsilon(\gamma_0). 
\end{equation*}
\end{theorem}

From this point forward, we deal with $\cap$-shaped pseudographs that are \emph{concave on the top}, i.e. their graph reparametrization $u$ is concave.  
The proof of Theorem \ref{thm:panettone} will rely on the intermediate result that for a fixed $\gamma_0 \in P_\psi$ there exists a $\gamma \in B_{\gamma_0} $ such that 
\begin{equation}\label{eq:tracelength}
\mathcal{L}(\gamma )  = \inf_{\tau \in B_{\gamma_0} } \mathcal{L}(\tau).
\end{equation}
We will call this minimization problem the \textit{Trace-Length} problem.  
The result will be deduced setting up a mesaure theoretic perimeter problem in $\mathbb{R}^2$, which can be solved using \cite{Ferriero}, a result about the convex hull of a finite perimeter set. Regularity of the minimizer also has to be shown. Let us point out that the result pairs up with certain results about the regularity of concave envelopes, examined in \cite{Oberman} and \cite{Kirchheim} to name only two out of many. From that we obtain that for each $\epsilon > 0$ we can find a $\cap$-shaped curve $\gamma^{(\epsilon)}$ such that $\mathcal{E}_\epsilon(\gamma^{(\epsilon)}) = \inf_{\gamma \in P_\psi} \mathcal{E}_\epsilon(\gamma)$. 

Theorem \ref{thm:panettone} paves the way for convex analysis techniques. We will employ these to bound the length of $P_\psi$-minimizers of the penalized functionals uniformly in $\epsilon$. This will turn out to be the main ingredient for the general existence result:

\begin{theorem} (Main Existence Theorem) \label{thm:main}

Let $\psi$ be an admissible obstacle (see Assumption \ref{ass:obst}). Then there exists a $\cap$-shaped pseudograph $\gamma \in P_\psi$ that is concave on the top and 
\begin{equation*}
\mathcal{E}(\gamma) = \inf_{ \tau \in P_\psi} \mathcal{E}(\tau) .
\end{equation*}

Moreover, set $\alpha := \inf_{\gamma \in P_\psi}  \mathcal{E}(\gamma)$ and define
\begin{equation*}
c_0 := \int_{\mathbb{R}} \frac{1}{(1+s^2)^\frac{5}{4}} ds , \quad   G(x) := \int_0^x \frac{1}{(1+s^2)^\frac{5}{4}} ds.
\end{equation*}
Then the following assertions hold true
\begin{enumerate} 
\item $\alpha \leq c_0^2$. 
\item If $\alpha < c_0^2$ then $\gamma$ satisfies 
\begin{align*}
\mathcal{L}(\gamma)&  \leq \max   \Bigg\lbrace  2\left( \sup_{x \in (0,1)} \psi(x) + ||\psi'||_{\infty} \right) + 1  , \\ &  \qquad \qquad \quad G^{-1} \left( \sqrt{\frac{\alpha + c_0^2}{2}} - \frac{c_0}{2} \right)  + \sqrt{1+ G^{-1} \left( \sqrt{\frac{\alpha + c_0^2}{2}} - \frac{c_0}{2} \right)^2 }   \Bigg\rbrace. 
\end{align*}
Additionally, $\gamma$ has a vertical tangent line at at most one of the two boundary points. 
\item In case that $\alpha = c_0^2$, define
\begin{equation}\label{def:Unull}
U_0(x) := \begin{cases} \frac{2}{c_0} \left( 1 + G^{-1} ( \frac{c_0}{2} - c_0 x )^2 \right)^{-\frac{1}{4}} & x \in (0,1) \\ 0 & x = 0,1 \end{cases}
\end{equation}
and $\gamma = \gamma_1 \oplus \gamma_2 \oplus \gamma_3 $ where 
\begin{eqnarray}\label{eq:vglkurve}
\gamma_1(t)   =&  (0 , t )^T  \quad &0 \leq t \leq S, \nonumber  \\
\gamma_2(t)  = &( t ,S +  U_0(t) )^T \quad &0 \leq t \leq 1,   \\
\gamma_3(t)  = &( 1, S - t )^T \quad &0 \leq t \leq S \nonumber, 
\end{eqnarray}
with $S := \sup_{x \in (0,1)} \psi(x) $ and $\oplus$ denoting the concatenation of curves. 
Then $\gamma$ possesses a weak reparametrization (in the sense of Definition \ref{def:d2}) lying in $P_\psi$ that is a minimizer of $\mathcal{E}$ with length
\begin{equation*}
\mathcal{L}(\gamma) = 2 S + \frac{1}{c_0}\int_{\mathbb{R}} \frac{1}{(1+t^2)^\frac{3}{4}} dt. 
\end{equation*}
\end{enumerate}
\end{theorem}
At this point there remains an open question: How can we find conditions that ensure that the minimizer is a graph? We know that in the case $\alpha < c_0^2$ at most one out of the two slopes at the boundary is infinite. For symmetric obstacles, we suspect that we can find a symmetric minimizer. Such a result would assure the graph property in the case $\alpha < c_0^2$. 
 \section{Preliminaries}
In this section we will derive or mention some results and easy estimates that will be useful for the work ahead. 

\subsection{Some Basic Facts in the Graph Case}
All the results listed here are proven in the appendix or in \cite{Anna}.

\begin{prop}\label{prop:p1} (Energy and oscillation of the derivative) 

Fix $u \in W^{2,2}(0,1)$. Then for each $b_1,b_2 \in [0,1]$ it holds that
\begin{equation*}
\mathcal{E}(u) \geq (G(u'(b_2)) - G(u'(b_1)))^2, 
\end{equation*}
where $G$ is defined as in Theorem \ref{thm:main}. 
\end{prop}

The following result is taken from \cite[Lemma 2.4]{Anna}.

\begin{prop}\label{prop:p3} (An upper bound for the least possible energy) 

Let $G_\psi$ be defined as in \eqref{eq:graph set}. Then 
\begin{equation*}
\inf_{u \in G_\psi} \mathcal{E}(u) \leq c_0^2. 
\end{equation*} 
\end{prop}

\begin{prop}\label{prop:p5} \label{prop:p6} \label{prop:p4}(Euler-Lagrange equation, properties of a minimizer)
 
Let $u \in G_\psi$ be such that 
\begin{equation*}
\mathcal{E}(u) = \inf_{w \in G_\psi} \mathcal{E}(w). 
\end{equation*}
Define 
\begin{equation*}
v(x) := \frac{u''(x) }{(1 + u'(x)^2)^\frac{5}{4}} = \frac{d}{dx} (G \circ u')(x). 
\end{equation*}
Then the following statements are true:
\begin{enumerate}
\item The minimizer $u$ is concave and $u \in C^2([0,1])$.
\item Away from the coincidence set, $u$ is smooth, that is $u \in C^\infty( \overline{\{ x \in (0,1) | u(x) > \psi(x) \}} )$ and  
\begin{equation*}
- \frac{v''(x)}{(1 + u'(x)^2)^\frac{3}{4}} + \frac{5}{2} \frac{\kappa(x) u'(x)}{(1 + u'(x)^2)^\frac{1}{4}} v'(x) = 0 \quad \forall x \in (0,1) : u(x) > \psi(x). 
\end{equation*}
\item If $u > \psi $ on some interval $(a,b)$ then there exists $C =C(a,b) \in \mathbb{R}$ such that 
\begin{equation*}
\frac{v'(x)}{(1 + u'(x)^2)^\frac{5}{4}} = C \quad \forall x \in (a,b) . 
\end{equation*}
\item $v(0) = v(1) = 0$. 
\item If $u > \psi $ on $[a,b] $ then 
\begin{equation*}
\max_{x \in [a,b] } | v(x) | = \max \{ |v(a)  | , |v(b)| \}. 
\end{equation*}
\item $v$ is decreasing in a neighborhood of $x_0 = 0$ and increasing in a neighborhood of $x_1 = 1$.
\item There is $x_0 \in (0,1) $ such that $u(x_0) = \psi(x_0)$. 
\end{enumerate}
\end{prop}

\begin{prop}\label{prop:p2} (Graphs with finite energy)
 
Let $u \in W^{2,1}_{loc}(0,1) $ be such that $u ' \in L^\infty(0,1)$ and $\mathcal{E}_\epsilon(u) < \infty $ for some $\epsilon \geq 0 $, where $\mathcal{E}_\epsilon$ is defined as in \eqref{eq:pen}. Then $u \in W^{2,2}(0,1).$ 
\end{prop}

\subsection{Elementary Examination of the Penalized Functional}
In what follows, we will examine $\mathcal{E}_\epsilon$, see \eqref{eq:pen}, using the very same methods as in Section 2 and 3 of \cite{Anna}. The computations will be provided in the appendix. 
\begin{prop}\label{prop:p7} (Energy for constant-velocity curves)

 Let $\gamma \in W^{2,1}_{loc}((0,1);\mathbb{R}^2) \cap W^{1,1}((0,1);\mathbb{R}^2)$ be immersed and $\epsilon > 0 $ be such that $\mathcal{E}_\epsilon(\gamma) < \infty$. Then the constant velocity reparametrization of $\gamma$ lies in $W^{2,2}((0,1);\mathbb{R}^2) $ and satisfies
\begin{equation*}
\mathcal{E}_\epsilon(\gamma) = \frac{1}{\mathcal{L}(\gamma)^3} \int_0^1 |\gamma''|^2 dt + \epsilon \mathcal{L}(\gamma). 
\end{equation*}
\end{prop}

\begin{prop}\label{prop:p8} (Existence of a minimizer) 

For each $\epsilon >0 $, $\mathcal{E}_\epsilon$ admits a minimizer in $P_\psi$ (see \eqref{eq:Pseudo} and Definition \ref{prop:p10}). 
\end{prop}

The following result is an analogue of Proposition \ref{prop:p4} (1) in the penalized case. 

\begin{prop}\label{prop:concavity} (Concavity of a graph minimizer) 

Assume that there is $u \in G_\psi$ such that 
\begin{equation*}
\mathcal{E}_\epsilon(u ) = \inf_{v \in G_\psi} \mathcal{E}_\epsilon(v) .
\end{equation*}
Then $u$ is concave. 
\end{prop}

\begin{prop}\label{prop:p9} (Euler-Lagrange equation for the penalized functional) 

Assume that there is $u \in G_\psi$ such that 
\begin{equation*}
\mathcal{E}_\epsilon(u) = \inf_{w \in G_\psi} \mathcal{E}_\epsilon (w). 
\end{equation*}
Define $v(x) = \frac{u''(x)}{(1+ u'(x)^2)^\frac{5}{4}}$. Then $u \in C^\infty ( \overline{\{ x \in (0,1) | u(x) > \psi(x) \}} ) $ and for all $x \in (0,1)$ such that $u(x) > \psi(x)$ it holds that 
\begin{enumerate}
\item 
\begin{equation*}
- 2\frac{v''(x)}{(1 + u'(x)^2)^\frac{5}{4}} +5\frac{u''(x) u'(x)}{(1 + u'(x)^2)^\frac{9}{4}} v'(x) = - \epsilon \frac{u''(x)}{(1+ u'(x)^2)^\frac{3}{2}} \geq 0, 
\end{equation*}
\item $v(0) = v(1) = 0$,
\item If $u > \psi $ on $(a,b)$ it holds that 
\begin{equation*}
\min_{x \in [a,b]} v(x) \geq \min( v(a), v(b) ).   
\end{equation*}
\end{enumerate}
\end{prop}
\begin{proof}
The proof follows the lines of  \cite[Proposition 3.2]{Anna} and \cite[Corollary 3.3]{Anna}. 
\end{proof}

\begin{prop}\label{prop:bpunkt} (Touching the obstacle)

Let $u \in G_\psi$ be such that 
\begin{equation*}
\mathcal{E}_{\epsilon}(u) = \inf_{w \in G_\psi} \mathcal{E}_{\epsilon}(w) .
\end{equation*}
Then there is $x_0 \in (0,1)$ such that $u(x_0) = \psi(x_0)$. 
\end{prop}

\section{Non-Existence for Large Cone Obstacles}

\begin{definition}(Symmetric cone obstacle)\label{def:symcone}

Let $A> 0 $ and $s \in (0, \frac{1}{2})$. 
We say $\psi$ is a \emph{symmetric cone obstacle} with \emph{valley} $[0,s]$ and \emph{peak} $A$ if 
\begin{equation*}
\psi(x) = \begin{cases} \frac{A}{\frac{1}{2}-s} (x-s) & 0 \leq x \leq \frac{1}{2} \\ \frac{A}{\frac{1}{2}-s} (1-s-x) & \frac{1}{2} \leq x \leq 1.
\end{cases} 
\end{equation*}

\end{definition}
Notice that any symmetric cone obstacle is admissible in the sense of Assumption \ref{ass:obst}. We will see that the properties of cone obstacles will lead to an explicit characterization of the contact set and finally to an explicit formula for candidates for minimizers. Eventually, a nonexistence result can be obtained. 

\begin{prop}\label{prop:cone1}(Touching cone obstacles)

Suppose that $u \in G_\psi $ is a minimizer of $\mathcal{E}$ in $G_\psi$ with respect to a symmetric cone obstacle $\psi$. Then $u(x) = \psi(x)$ if and only if $x = \frac{1}{2}$. 
\end{prop}
\begin{proof}
First recall from Proposition \ref{prop:p4} that $u \in C^2([0,1]) $ is concave and $u$ has to touch the obstacle somewhere. Suppose that $u$ touches $\psi$ at some $0 < a < \frac{1}{2} $. Then $\psi$ is continuously differentiable in a neighborhood of $a$ and $u - \psi $ has a local minimum at $a$, i.e. $u'(a) = \psi'(a)$. But then, concavity of $u$ and the cone property of $\psi$ imply
\begin{equation*}
0 = u(0) \leq u(a) + u'(a) (0 - a) = \psi(a) + \psi'(a) (0-a) = \psi(0) < 0, 
\end{equation*}
a contradiction.
  In case that $u$ touches the obstacle at some $a > \frac{1}{2}$, note that $u (1- \cdot) \in G_\psi$ is another minimizer for which the arguments above can be repeated to obtain the very same contradiction. 
\end{proof}

\begin{lemma}\label{lem:shapeofmin} (An explicit formula for a graph minimizer) 
 
Let $u \in G_\psi$ be a minimizer of $\mathcal{E} $ in $G_\psi$. Then there are constants $C_0 \leq 0 \leq C_1$ such that 
\begin{equation}\label{eq:expli}
u'(x) = \begin{cases} F_0^{-1}( \sqrt{2|C_0| }x) & 0 \leq x \leq \frac{1}{2} \\  F_1^{-1} ( \sqrt{2 C_1  } (1-x) ) & \frac{1}{2} \leq x \leq 1 \end{cases}, 
\end{equation}
where 
\begin{equation*}
F_0(x) = \int_{x}^{u'(0)} \frac{1}{\sqrt{u'(0) - z} (1 + z^2)^\frac{5}{4}} dz ,
\end{equation*}  
\begin{equation*}
F_1(x) = \int_{u'(1)}^{x} \frac{1}{\sqrt{z- u'(1)}} \frac{1}{(1+z^2)^\frac{5}{4}} dz .
\end{equation*}
\end{lemma}
\begin{proof}
Since according to Proposition \ref{prop:cone1}, the only point of contact with the obstacle is at $x = \frac{1}{2}$, we can find by Proposition \ref{prop:p5} $(3)$ some $C_0,C_1 \in \mathbb{R}$ such that 
\begin{equation*}
\frac{v'(x)}{(1+ u'(x)^2)^\frac{5}{4}} = C_0 \quad 0 \leq x < \frac{1}{2}
\end{equation*}
and 
\begin{equation*}
\frac{v'(x)}{(1+ u'(x)^2)^\frac{5}{4}} = C_1 \quad \frac{1}{2} < x \leq 1,
\end{equation*}
where $v$ is defined as in Proposition \ref{prop:p5}.
The signs of the constants are due to the fact that by Proposition \ref{prop:p5} $(6)$,  $v$ is decreasing near $0$ and increasing near $1$.    
For $0 \leq x \leq \frac{1}{2}$ one can multiply $v'(x) = C_0(1 + u'(x)^2)^\frac{5}{4}$ by $v(x)= \frac{u''(x)}{(1+ u'(x)^2)^\frac{5}{4}}$ to obtain 
\begin{equation*}
v(x) v'(x) = C_0 u''(x) \quad 0 \leq x \leq \frac{1}{2}.
\end{equation*}
Integration yields that 
\begin{equation*}
v(x)^2 = 2 C_0( u'(x) - u'(0) )  + D .
\end{equation*}
However, evaluating the expression at $x= 0 $, we can conclude with Proposition \ref{prop:p5} $(4)$ that $ D= 0 $. Recalling that $u$ is concave:
\begin{equation}\label{eq:diffgl}
u''(x)^2 = 2 C_0( u'(x) - u'(0) ) ( 1 + u'(x)^2)^\frac{5}{2}= 2 |C_0| (u'(0) - u'(x) ) ( 1 + u'(x)^2)^\frac{5}{2}.
\end{equation}
Again by concavity, there is only one possible choice of signs to make \eqref{eq:diffgl} hold true, namely when
\begin{equation*}
-u''(x) = \sqrt{2|C_0|} \sqrt{u'(0)- u'(x)} (1 + u'(x)^2)^\frac{5}{4}.
\end{equation*}
Now there are two cases: Either, there is $\delta \in ( 0, \nicefrac{1}{2} )  $ such that $v_{\mid_{[0,\delta]}} \equiv 0 $ or $v_{\mid_{(0,\frac{1}{2})}} < 0 $. Indeed, if there is $\delta \in (0, \nicefrac{1}{2})$ such that $v(\delta) = 0 $ then  Proposition \ref{prop:p5} (5) vields that $ v \equiv 0 $ on $[0, \delta]$.   In the first case,  $C_0 = 0$ and since $F_0^{-1}(0) = u'(0) $, \eqref{eq:expli} holds true. In the remaining case, for each $\epsilon > 0 $  we can solve the ODE with separation of variables to obtain
\begin{equation}\label{eq:expli2}
F_0(u'(x)) - F_0(u'(\epsilon) )  = \sqrt{2|C_0|  } (x-\epsilon) ,  \quad 0 < x \leq \frac{1}{2}.
\end{equation}
Notice that 
\begin{align*}
\lim_{\epsilon \rightarrow 0 } F_0(u'(\epsilon) )  & = \lim_{\epsilon \rightarrow 0 } \int_{u'(\epsilon)}^{u'(0)} \frac{1}{\sqrt{u'(0)- z} (1+ z^2)^\frac{5}{4}} dz 
\\ & \leq \limsup_{\epsilon \rightarrow 0 } \int_{ u'(\epsilon) }^{u'(0)} \frac{1}{\sqrt{u'(0) -z}} dz \\ & = 2 \sqrt{u'(0) - u'(\epsilon)} \rightarrow 0 \quad ( \epsilon \rightarrow 0 ) 
\end{align*}
because $W^{2,2}(0,1)$ embeds compactly into $C^1([0,1])$. Letting $\epsilon \rightarrow 0 $ in \eqref{eq:expli2} proves the claim. 
Very similarly we obtain the formula given for $\frac{1}{2} \leq x \leq 1.$ 
\end{proof}

\begin{proof}[Proof of Theorem 1.1]
Assume for a contradiction that $\inf_{v \in M} \mathcal{E}(v)$ is attained by $u \in G_\psi$. Without loss of generality we can assume that $u'(\frac{1}{2}) \leq 0 $. Indeed, if $u' ( \frac{1}{2} ) > 0 $, notice that $u(1-\cdot)$ is also a minimizer satisfying the same $L^\infty(0,1)$-bounds.

Now choose $F_0$ as in Lemma \ref{lem:shapeofmin}. It must hold that $C_0\neq 0$ since otherwise $u'$ is constant on $[0, \frac{1}{2}]$ and there is no way for $u'(\frac{1}{2})$ to be nonpositive. Choose $x^*$ to be the smallest point in $(0,1)$ at which $u$ attains its maximum $||u||_\infty$. Observe that $x^* \leq \frac{1}{2}$ since $u$ is $C^2([0,1])$, concave, $u'(x^*) = 0 $ and $u'(\frac{1}{2}) \leq 0 $. 
Note that 
\begin{eqnarray}\label{eq:abschtzung}
\frac{\sqrt{2|C_0|}}{2} = F_0 ( u'( \nicefrac{1}{2} ) ) &  = &  \int_{u'(\frac{1}{2})}^{u'(0)} \frac{1}{\sqrt{u'(0)- t}} \frac{1}{(1+t^2)^\frac{5}{4}} dt  \nonumber \\ & \geq  & \int_{0}^{u'(0)} \frac{1}{\sqrt{u'(0)- t}} \frac{1}{(1+t^2)^\frac{5}{4}} dt .
\end{eqnarray}
 Using this estimate we find
\begin{eqnarray*}
||u||_\infty & =& \int_0^{x^*} u'(s) ds  =  \int_0^{x^*} F_0^{-1} ( \sqrt{2|C_0|} s) ds \\ &= & \frac{1}{\sqrt{2 |C_0|}} \int_0^{\sqrt{2 |C_0|} x^*} F_0^{-1}(s) ds =  \frac{1}{\sqrt{2|C_0|}} \int_{F_0^{-1} (0) }^{F_0^{-1} ( \sqrt{2|C_0|}x^*)} t F_0'(t) dt \\ &= &  \frac{1}{\sqrt{2|C_0|}} \int_{u'(0) }^{u'(x^*)} t F_0'(t) dt = \frac{1}{\sqrt{2 |C_0|}} \int_0^{u'(0)} \frac{t}{\sqrt{u'(0)- t}} \frac{1}{(1+ t^2)^\frac{5}{4}} dt \\ & \leq & \frac{\int_0^{u'(0)} \frac{t}{\sqrt{u'(0)- t}} \frac{1}{(1+ t^2)^\frac{5}{4}} dt}{2 \int_0^{u'(0)} \frac{1}{\sqrt{u'(0)- t}} \frac{1}{(1+ t^2)^\frac{5}{4}} dt}.
\end{eqnarray*}
Using Lemma \ref{lem:intiden} we conclude that 
\begin{equation*}
||u||_\infty \leq \frac{1}{3} u'(0) \frac{\mathit{HYP2F1}(1, \frac{3}{2} ; \frac{7}{4}, -u'(0)^2 )}{\mathit{HYP2F1}(\frac{1}{2}, 1; \frac{3}{4}, -u'(0)^2)},
\end{equation*}
but this is a contradiction to 
\begin{equation}\label{eq:haeslsup}
 \sup_{A> 0 } \frac{1}{3}A \frac{\mathit{HYP2F1}( 1, \frac{3}{2} ; \frac{7}{4} , -A^2) }{\mathit{HYP2F1}(\frac{1}{2},1, \frac{3}{4}, -A^2)} < \sup_{x \in (0,1)} \psi(x)  \leq  ||u||_\infty,
\end{equation}
provided that we can show that the given supremum is finite. We will show the finiteness in two steps, one of which will be done in the appendix. The first step will be to show that on every compact subset $K$ of $[0, \infty) $ 
\begin{equation*}
\sup_{A \in K } \frac{1}{3}A \frac{\mathit{HYP2F1}( 1, \frac{3}{2} ; \frac{7}{4} , -A^2) }{\mathit{HYP2F1}(\frac{1}{2},1, \frac{3}{4}, -A^2)} < \infty,
\end{equation*}
the second one is that
\begin{equation*}
\lim_{A \rightarrow \infty} \frac{1}{3}A \frac{\mathit{HYP2F1}( 1, \frac{3}{2} ; \frac{7}{4} , -A^2) }{\mathit{HYP2F1}(\frac{1}{2},1, \frac{3}{4}, -A^2)}
\end{equation*}
exists and is finite. For the first step, consider the following estimate:
\begin{equation*}
\frac{1}{3}A \frac{\mathit{HYP2F1}( 1, \frac{3}{2} ; \frac{7}{4} , -A^2) }{\mathit{HYP2F1}(\frac{1}{2},1, \frac{3}{4}, -A^2)} =  \frac{\int_0^{A} \frac{t}{\sqrt{A- t}} \frac{1}{(1+ t^2)^\frac{5}{4}} dt}{2 \int_0^{A} \frac{1}{\sqrt{A- t}} \frac{1}{(1+ t^2)^\frac{5}{4}} dt} \leq \frac{A}{2}
\end{equation*}
which we obtained using the estimate 
\begin{equation*}
\int_0^{A} \frac{t}{\sqrt{A- t}} \frac{1}{(1+ t^2)^\frac{5}{4}} dt \leq  A \int_0^{A} \frac{1}{\sqrt{A- t}} \frac{1}{(1+ t^2)^\frac{5}{4}} dt.
\end{equation*}
 For the second step we refer to the Appendix, see Lemma \ref{lem:finiteness}.  
\end{proof}
\begin{remark}
Uniqueness of a minimizer is a very interesting problem. As an intermediate step, one could attempt to show symmetry of a minimizer, possibly using rearrangement inequalities. Indeed, if it is at all possible to find a non-symmetric minimizer $u\in W^{2,2}(0,1)$, then it cannot be unique, since $u(1-\cdot)$ would also be a minimizer. 
\end{remark}

\section{The Pseudograph Framework}\label{sec:Pseudograph}

In the next section, we will examine another minimization problem, more precisely show existence and regularity for the \textit{trace-length} problem. We will see at the very end of this section, how these results can be used to understand the shape of a minimizer of the original problem. The main issue with regularity in this section is that we have to consider obstacles that are possibly nonnegative at the boundary. Existence is also an issue, since we consider a second order problem that is invariant with respect to reparametrization and therefore the problem might be ill-posed in any reflexive Sobolev space.

 \subsection{Pseudographs and Weak Reparametrizations}
 The following concepts will be needed when it comes to the trace-length problem (see \eqref{eq:tracelength}) A natural space for this problem is $W^{1,1}$ (or even $BV$). In this space, the reparametrizations we have to consider are $W^{1,1}$, so this section is dedicated to understanding these kinds of reparametrizations. The proofs are again to be found in the appendix.  

\begin{prop}\label{prop:p12} (The vertical set) 

Let $\gamma$ be a pseudograph such that $\gamma_1(0) = 0, \gamma_1(1) = 1$. Then $\gamma_1 $ satisfies $\mathcal{H}^0( \{ t \in (0,1) | \gamma_1(t) = x \} ) = 1 $ for (Lebesgue-)almost every $x \in (0,1)$, where $\mathcal{H}^0$ is the zero-dimensional Hausdorff measure. 
\end{prop}

\begin{definition}\label{def:d1} (Invertibility in $W^{1,1}$) 

A function $u \in W^{1,1}(a,b)$ such that $u(a)= c , u(b) = d $ is called \emph{invertible in $W^{1,1}(a,b)$} if $u$ is bijective and $u^{-1} \in W^{1,1}(c,d)$ if $c< d$ or $u^{-1} \in W^{1,1} (d,c) $ if $d < c$.
\end{definition}

\begin{remark}
Functions that lie in $W^{1,1}$ are uniformly continuous. If we further require that they are invertible they have to be either increasing or decreasing. Without loss of generality we consider increasing functions only, the other case is similar. 
\end{remark}

\begin{prop}\label{prop:p15} (Luzin-N-Property)

Assume that $v$ is invertible in $W^{1,1}$. Then $v$ and $v^{-1}$ have the Luzin-N-property, i.e. $v, v^{-1} $ map sets of Lebesgue measure zero to sets of Lebesgue measure zero. 
\end{prop}

\begin{prop}\label{prop:p11} (Derivative Formula) 

Let $v : (a,b) \rightarrow (v(a),v(b))$ be invertible in $W^{1,1}$ and increasing . Then, $v' \neq 0 $ a.e., $v' \circ  v^{-1}$ is measurable and for almost every $x \in (v(a),v(b)) $ :  $(v^{-1})'(x) = \frac{1}{v'(v^{-1}(x))} $.  
\end{prop}

\begin{prop}\label{prop:p14} (Composition by diffeomorphisms) 

Let $v \in C^1(a,b)$ be such that there exists $\epsilon > 0$ with $\epsilon < v' < \frac{1}{\epsilon}$ .  If $u$ is invertible in $W^{1,1}(v(a),v(b))$ then $ u \circ v $ is invertible in $W^{1,1}(a,b)$.  

\end{prop}

\begin{prop}\label{prop:p16} (Left-Composition with Sobolev functions) 

Let  $v$ be invertible in $W^{1,1}(a,b)$, increasing and $u \in W^{1,1} (v(a),v(b)) $. Then $u \circ v \in W^{1,1}(a,b)$ with derivative $(u \circ v)'(x) = u'(v(x))v'(x)$. 
\end{prop}
\begin{definition}\label{def:d2} (Weak Reparametrization) 

Let $\gamma \in W^{1,1} ((a,b); \mathbb{R}^2) $ and $\iota \in W^{1,1} ((c,d); \mathbb{R}^2) $ be two curves. We say that $\gamma$ is a \emph{(weak) reparametrization} of $\iota$ if there exists $\phi$ invertible in $W^{1,1}$  such that $\gamma \circ \phi = \iota$. 
\end{definition}

\begin{prop}\label{prop:p18} (Arclength reparametrization of graphs)

Let $u \in C^1(0,1)\cap W^{1,1}(0,1)$. Then there exists $L> 0 $ and $\phi \in W^{1,1}(0,L) \cap C^1(0,L)$ invertible in $W^{1,1}$ such that 
\begin{equation*}
\begin{cases}
\phi'(s) = \frac{1}{\sqrt{1+ u'(\phi(s))^2}} \quad s \in (0,L), \\
\phi(0) = 0,  \\ \phi(L) = 1 .
\end{cases}
\end{equation*}
Moreover, $s \mapsto (\phi(s), u(\phi(s))) $ is a weak reparametrization of $(x,u(x))$. Additionally, if $u \in W^{2,2}_{loc}(0,1)$, then $\phi \in W^{2,2}_{loc}(0,L) $ and if $u$ is concave, then the (weak) reparametrization $s \mapsto ( \phi(s), u(\phi(s) ) $ lives in $W^{2,1}(0,L) \cap W^{2,2}_{loc} (0,L) $. 
\end{prop}

\subsection{Regularity for some variational problems}

We shall see later, that a minimizing sequence for the trace-length problem will have a limit in $BV(0,1)$. However, we need more regularity to assure that such a minimizer is admissible for $\mathcal{E}_\epsilon$. Luckily, we will see that a minimizer of the trace-length problem is also concave. This will imply $W^{1,1}$-regularity of the minimizer. Further regularity for the length problem for graphs is investigated at the end of this subsection. The proof is technical and to be found in the appendix.
 
\begin{lemma}\label{lem:l1} (Regularity for problems in $BV(0,1)$) 

Assume $u \in BV(0,1)\cap W^{1,p}_{loc}(0,1) $ for some $p \in (1, \infty]$. Then $u \in W^{1,1}(0,1)$. Especially, if $u \in BV(0,1)$ is concave, then $u \in W^{1,1}(0,1)$.
\end{lemma}

%
%

\begin{lemma}\label{lem:rgulang} (Regularity for the length problem) 

 Assume that $\psi \in W^{1,\infty}(a,b)$ is such that $\psi'' \in L^2_{loc}(a,b)$, and $d_1, d_2 \in \mathbb{R}$ are such that $\psi(a) < d_1 $ and $\psi(b) < d_2 $. Define
\begin{equation*}
M_1 := \left\lbrace  v \in W^{1, \infty}(a, b) , v(a) = d_1 , v(b) = d_2 , v(x) \geq \psi(x) , \; x \in (a,b) \right\rbrace.
\end{equation*} 
Assume that $u \in M_1$ is such that 
\begin{equation*}
\int_{a}^{b} \sqrt{1+ u'^2} dx = \inf_{ v \in M_1} \int_{a}^{b} \sqrt{1+ v'^2} dx .
\end{equation*}
Then $u$ is the unique solution of 
\begin{equation}\label{eq:laengenvarungl}
\int_{a}^{b} \frac{u'}{\sqrt{1+ u'^2}} ( v' - u') dx \geq 0 \quad \forall v \in M_1 . 
\end{equation}
Additionally, $||u'||_\infty \leq \max\{||\psi'||_\infty, \frac{|d_2- d_1| }{b - a} \} $ and $u'' \in L^2_{loc}(a,b) $.  
\end{lemma}

\subsection{The Trace-Length Problem}
\begin{definition}($\cap$-shape)\label{defi:panettone}

Fix $\gamma_0 \in P_\psi$. We define 
\begin{align*}
M_{\gamma_0} := & \{ \gamma \in W^{1,1}((0,1);\mathbb{R}^2) \; | \;   \gamma(0)=(0,0)^T , \gamma(1) = (1,0)^T, \\* & \qquad \exists \; 0 \leq \beta_1 < \beta_2 \leq 1 : (\gamma_1)_{\mid_{ [0, \beta_1]}} = 0, (\gamma_1)_{\mid_ {[\beta_2,1]}} = 1 , \\* & \qquad (\gamma_1)_{\mid_{(\beta_1,\beta_2)}} \; \textrm{is invertible in $W^{1,1}$ and on $(\beta_1,\beta_2)$} \\* & \qquad  
 u:= \gamma_2 \circ \gamma_1^{-1} \; \textrm{satisfies } u \circ \gamma_{0,1} \geq \gamma_{0,2} \}  
\end{align*}
where by $\gamma_1^{-1}$ in the last line we mean the inverse on $(\beta_1,\beta_2)$ and call elements of $M_{\gamma_0}$ \emph{$\cap$-shaped curves} lying above $\gamma_0$. Any $\gamma \in P_\psi \cap M_{\gamma_0}$ is called \emph{pseudograph-$\cap$-shape} above $\gamma_0$ and we denote the set of all such curves by $B_{\gamma_0} $. If $u$ is concave, we call $\gamma$ \emph{concave on the top}. The function $u$ is called the \emph{graph reparametrization} associated to $\gamma$.
 \end{definition}

From now on, we fix a pseudograph $\gamma_0 \in W^{1,1}((0,1);\mathbb{R}^2) \cap W^{2,1}_{loc}( (0,1); \mathbb{R}^2) $ and intend to show that there exists $\gamma \in M_{\gamma_0} $ such that $\mathcal{L}(\gamma) = \inf_{\tau \in M_{\gamma_0}} \mathcal{L}(\tau)$. As it turns out, such a $\gamma$ lies in $B_{\gamma_0}$  (defined as in Theorem \ref{thm:panettone}) and will be the curve constructed in Theorem \ref{thm:panettone}. We will first show the regularity and postpone existence of $\gamma$. 

\subsection{Regularity of a Minimizer of the Trace-Length Problem}

\begin{prop}\label{prop:vartop}(A variational inequality on the top) 

Let $\gamma \in M_{\gamma_0} $ be a minimizer of $\mathcal{L}$ in $M_{\gamma_0} $ with $\beta_1, \beta_2 $ and $u$ be as in Definition \ref{defi:panettone}. Then for each nonnegative $\phi \in C_0^\infty(0,1) $  it holds that
\begin{equation}\label{eq:vartop}
\int_0^1 \frac{u'}{\sqrt{1+u'^2}} \phi' dx \geq 0 .
\end{equation}
Additionally, the set $U :=  \{ x \in (0,1) | (x,u(x)) \not \in  \gamma_0([0,1]) \}  $ is open, $u \in C^\infty(U) $  and satisfies $u'' = 0 $ in $U$.   
\end{prop}
\begin{proof}
Let $\phi \in C_0^\infty(0,1) $ be arbitrary. Define for $s \in \mathbb{R}$ 
\begin{equation*}
\tau_s(t) := \begin{cases} \gamma(t) & t \in [0, \beta_1] \\ (\gamma_1(t) , (u + s\phi)(\gamma_1(t)))^T  & t \in (\beta_1, \beta_2) \\ \gamma(t) & t \in [\beta_2,1] \end{cases}.
\end{equation*}
For the first part of the claim, assume that $\phi \geq 0 $. Let us check that $\tau_s \in M_{\gamma_0} $ for each $s \geq 0 $.
First of all $\tau_s \in W^{1,1}((0,1);\mathbb{R}^2)$ because of \cite[Section 4.2.2]{Evansgariepy}. The remaining conditions are straightforward to check. 
 
Now 
\begin{eqnarray*}
0  & \leq &  \frac{d}{ds}_{\mid_{s = 0 }} \mathcal{L}(\tau_s) \\ & = &  \frac{d}{ds}_{\mid_{s = 0 }} \left( \gamma_2(\beta_1)+ \gamma_2(\beta_2) + \int_0^1 \sqrt{1 + (u'(x) + s \phi'(x))^2} dx \right) \\ & =& \int_0^1 \frac{u'(x)}{\sqrt{1+ u'(x)^2}} \phi'(x) dx ,
\end{eqnarray*}
which proves the first part of the claim. 

For the second part we take $\phi \in C_0^\infty(U)$ with arbitrary sign. Notice that there exists $\epsilon > 0 $ such that for $|s | < \epsilon$ it holds that $\tau_s \in M_{\gamma_0}$.  This is true since continuity of $u$, compactness of $\mathrm{supp}(\phi)$, and compactness of $\gamma_0([0,1])$ yield that $\mathrm{dist}( u(\mathrm{supp} \phi ), \gamma_0([0,1]) ) > 0 $. With a computation similar to the last one we obtain 
\begin{equation*}
0 = \frac{d}{ds}_{\mid_{s = 0 }} \mathcal{L}(\tau_s) = \int_0^1 \frac{u'(x)}{\sqrt{1+ u'(x)^2}} \phi'(x) dx.
\end{equation*}
Therefore $\frac{u'}{\sqrt{1+ u'^2}}$ is constant almost everywhere on every connected component of $U$ and since $ z \mapsto \frac{z}{\sqrt{1+z^2}}$ is injective, it follows that $u'$  is constant almost everywhere on every connected component of $U$. Equivalently, $u$ has a $C^\infty(U)$-representative and $u '' \equiv 0 $ a.e. on $U$. 
\end{proof}

\begin{prop}\label{prop:c1regu} ($C^{1}$-regularity on the top)

Let $\gamma \in M_{\gamma_0} $ be a minimizer with associated graph reparametrization $u$. Then $u \in C^1(0,1)$. Moreover, for each $t \in (0,1)$ such that $u( \gamma_{0,1}(t))  = \gamma_{0,2}(t) $ it must hold that $\gamma_{0,1}'(t) > 0 $. 
\end{prop}
\begin{proof}
Thanks to Proposition \ref{prop:vartop} and \cite[Section 1.8, Corollary 1]{Evansgariepy}, there is a Radon measure $\mu$ in $\mathcal{B}(0,1)$ such that 
\begin{equation*}
\int_0^1 \frac{u'(s) }{(1+ u'(s)^2)^\frac{1}{2}} \phi'(s) ds = \int_0^1 \phi  \; d\mu \quad \forall \phi \in C_0^\infty(0,1) .
\end{equation*}
Now fix $\epsilon > 0 $. Since $\mu$ is Radon, it holds that $\mu ( (\epsilon, 1- \epsilon) ) < \infty $ and in case that $\phi \in C_0^\infty (\epsilon, 1- \epsilon) $ we can rearrange the right hand side using Fubini's theorem: 
\begin{equation*}
\int_{\epsilon}^{1- \epsilon} \phi d\mu = \int_{\epsilon}^{1 - \epsilon } \left( \int_{\epsilon}^x \phi'(s) ds \right) d\mu(x) = \int_{\epsilon}^{1- \epsilon } \int_s^{1-\epsilon} d\mu(x) \phi'(s) ds.  
\end{equation*}
Eventually, 
\begin{equation*}
\int_{\epsilon}^{1- \epsilon} \frac{u'(s) }{(1+ u'(s)^2)^\frac{1}{2}} \phi'(s) ds = \int_\epsilon^{1- \epsilon} \mu((s,1- \epsilon)) \phi'(s) ds \quad \forall \phi \in C_0^\infty(0,1) 
\end{equation*}
such that almost everywhere on $(\epsilon, 1- \epsilon)$ 
\begin{equation}\label{eq:gl11}
\frac{u'(s)}{\sqrt{1+ u'(s)^2}} = C_0 +  \mu ((s, 1- \epsilon) ),
\end{equation}
for some constant $C_0 \in \mathbb{R}$. 
Since the mapping $ z \mapsto \frac{z}{\sqrt{1+z^2}}$ is strictly monotone, we can infer, after choice of a representative, that $u'$ is decreasing and therefore has left- and right-sided limits satisfying $u'(s+ 0 ) \leq u'(s-0)$ almost everywhere for each $s \in (\epsilon, 1- \epsilon)$. Choosing $\epsilon = \frac{1}{m}$ and taking the countable union we obtain that $u'$ is decreasing a.e. on $(0,1)$ and $u'(s+0) \leq u'(s-0)$ for each $s \in (0,1)$. Let $U$ be as in Proposition \ref{prop:vartop}. If $s \in U$, the limits coincide because of the second part of the very same proposition. Our goal next is to derive that 
\begin{equation}\label{eq:unglinbeidesein}
u'(\gamma_{0,1}(t) + 0) \gamma_{0,1}'(t) \geq \gamma_{0,2}'(t) \geq u'(\gamma_{0,1}(t) - 0) \gamma_{0,1}'(t) \quad \forall t : u( \gamma_{0,1}(t) ) = \gamma_{0,2}(t) .
\end{equation}  
Using that $\gamma_{0,1} $ is $C^1(0,1)$ and therefore locally Lipschitz, we can use the version of the coarea formula provided in \cite[Section 3.4.3]{Evansgariepy} to find for each fixed $t \in (0,1) \setminus U$: 
\begin{align}\label{eq:2}
& u'(\gamma_{0,1}(t) + 0 ) \gamma_{0,1}'(t)  =  \lim_{h \rightarrow 0+} \frac{1}{h} \int_t^{t+h} u'(\gamma_{0,1}(s)) \gamma_{0,1}'(s) ds \nonumber  
\\ &   \qquad = \lim_{h \rightarrow 0 +} \frac{1}{h} \int_{\gamma_{0,1}(t)}^{\gamma_{0,1}(t+h)}  u'(w) \mathcal{H}^0(\gamma_{0,1}^{-1} (\{ w \} ) ) dw \nonumber =  \lim_{h \rightarrow 0 + } \frac{1}{h} \int_{\gamma_{0,1}(t)}^{\gamma_{0,1}(t+h)}  u'(w)  dw \nonumber  \\ & \qquad  =  \lim_{h \rightarrow 0 +} \frac{u(\gamma_{0,1}(t+h) ) - u(\gamma_{0,1}(t))}{h} \nonumber   =  \lim_{h \rightarrow 0 +} \frac{u(\gamma_{0,1}(t+h) ) - \gamma_{0,2}(t)}{h} \nonumber \\  &  \qquad \geq  \liminf_{h \rightarrow 0 +} \frac{\gamma_{0,2}(t+h)- \gamma_{0,2}(t) }{h} = \gamma_{0,2}'(t).
\end{align}
The remaining inequality in \eqref{eq:unglinbeidesein} can be shown similarly. If we now assume that $\gamma_{0,1}'(t)=0 $ then \eqref{eq:unglinbeidesein} would imply that $\gamma_{0,2}'(t) = 0$ which is a contradiction to the immersedness of $\gamma_0$. Therefore $\gamma_{0,1}'(t) > 0 $ and $u'(\gamma_{0,1}(t) + 0) \geq u'(\gamma_{0,1}(t) - 0 )$. Together with the arguments after \eqref{eq:gl11}, we find that left-sided limit and right sided limit do indeed coincide almost everywhere, even on $(0,1) \setminus U$. This shows the desired regularity result. For the rest of the claim, recall that on the road, right after \eqref{eq:2}, we found that $\gamma_{0,1}'(t) \neq 0 $ for each $t$ such that $u( \gamma_{0,1}(t) )  = \gamma_{0,2}(t) $.  
\end{proof}
\begin{cor}\label{cor:conc} (Concavity on the top) 

Let $\gamma \in M_{\gamma_0}$ be a minimizer of $\mathcal{L}$ with associated graph $u$. Then $u$ is concave.
\end{cor}
\begin{proof}
Recall that, according to Proposition \ref{prop:c1regu}, $u' \in C^0(0,1)$. Now, \eqref{eq:gl11} yields that for each $\epsilon > 0 $ , $ x \mapsto \frac{u'(x)}{\sqrt{1+ u'(x)^2}}$ is decreasing  in $(\epsilon, 1- \epsilon) $. Therefore, since $z \mapsto \frac{z}{\sqrt{1+z^2}}$ is increasing, $u'$ is decreasing. Choose $x, y \in (0,1) $ such that $x > y$ and observe that 
\begin{equation}\label{eq:4.15}
u(x) - u(y) = \int_y^x u'(s) ds \leq u'(y) (x-y) .
\end{equation}
A very elementary computation shows that \eqref{eq:4.15} implies concavity.
\end{proof}
\begin{lemma}\label{lem:localregu}(Local $W^{2,2}$-regularity on the top)
 
Let $\gamma \in M_{\gamma_0} $ be a minimizer of $\mathcal{L}$ with associated graph reparametrization $u$ and $\beta_1, \beta_2$ as in Definition \ref{defi:panettone}. Assume that $t \in (\beta_1, \beta_2) $ is such that $ u\circ \gamma_{0,1}(t) = \gamma_{0,2}(t)$. 

Then one of the following assertions is true:
\begin{enumerate}
\item $\gamma_{\mid_{(\beta_1+ \delta,\gamma_1^{-1}( \gamma_{0,1}(t)))}}$ is a (weak) reparametrization of $\gamma_0$ for each $\delta >0 $.
\item $\gamma_{\mid_{(\gamma_1^{-1}( \gamma_{0,1}(t)) ,\beta_2- \delta)}}$ is a (weak) reparametrization of $\gamma_0$ for each $\delta > 0 $.  
\item There exist  $0 < x_1'< \gamma_{0,1}(t) < x_2' < 1$ and $\beta_1 < t_1 < t_2 < \beta_2 $ such that the restriction $\gamma_{0,1} : ( t_1 , t_2) \rightarrow (x_1' , x_2') $ is a diffeomorphism, $x_1' , x_2' $ are not points of contact, and  $u \in W^{2,2}_{loc} (x_1', x_2') $. 
\end{enumerate} 
\end{lemma}
\begin{proof}
 We start by showing that if $(1)$ is not true then there has to be a point $p \in (0,t )$ such that $\gamma_{0,1}(p) > 0$ and $u(\gamma_{0,1} (p) ) > \gamma_{0,2}(p)$. 

Indeed, assume that $(1)$ is not true and $u( \gamma_{0,1}(p) ) = \gamma_{0,2}(p)$ for all $p \in (0,t)$ such that $\gamma_{0,1}(p) > 0 $. As shown in Proposition \ref{prop:c1regu}, we find that for each $s \in (\beta_1 , t] $  it holds that
\begin{equation}\label{eq:4.16}
 \gamma_{0,1}(s) > 0 ,  \; \;  \gamma_{0,1}'(s) >0 , \; \;    u ( \gamma_1 ( \gamma_1^{-1} \circ \gamma_{0,1} (s) ) ) = \gamma_{0,2}(s),
 \end{equation}
 and thus $ \gamma_2 (  \gamma_1^{-1} \circ \gamma_{0,1} (s) ) = \gamma_{0,2}(s) $. 
 Now, (1) is not true if and only if $\gamma_1^{-1} \circ \gamma_{0,1} $ is not a weak reparametrization, i.e. it does not have sufficient regularity, but this is assured by Proposition \ref{prop:p14}, at least if we restrict to $(\beta_1 + \delta , \gamma_1^{-1} ( \gamma_{0,1}(t) ) )  $.

Similarly, the fact that $(2)$ fails to hold true implies that we can find $q \in (t, 1) $ such that $\gamma_{0,1} (q) < 1 $ and  $u( \gamma_{0,1}(q) ) > \gamma_{0,2}(q) $. 

Now we assume that $(1)$ and $(2)$ do not hold true. We have to show that $(3)$ does.
 
 Since for each $t$ on the contact set $\gamma_{0,1}'(t) \neq 0 $ and $\gamma_{0,1} \in C^1(0,1)$ , there is an open neighborhood of the contact set such that $\gamma_{0,1}' \neq 0 $ on this neighborhood. Taking a connected component of this neighborhood, we can infer the existence of $x_1' , x_2' , t_1, t_2 $ that are not points of contact and such that $\gamma_{0,1} : (t_1, t_2) \rightarrow (x_1',x_2')$ is a diffeomorphism. 

Since $u \in C^1(0,1)$ we obtain that $u \in W^{1, \infty} (x_1' , x_2') $. For the arguments to come, define $d_1 := u(x_1') $ and $d_2 = u(x_2')$. Further, define
\begin{equation*}
M_1 := \left\lbrace  v \in W^{1, \infty}(x_1', x_2') , v(x_1') = d_1 , v(x_2') = d_2 , v(x) \geq \gamma_{0,2} \circ \gamma_{0,1}^{-1} (x), \; x \in (x_1',x_2') \right\rbrace.
\end{equation*} 
We claim that then
\begin{equation*}
\int_{x_1'}^{x_2'} \sqrt{1+ u'^2} dx = \inf_{ v \in M_1} \int_{x_1'}^{x_2'} \sqrt{1+ v'^2} dx 
\end{equation*}
from which it follows according to Lemma \ref{lem:rgulang} that $u\in W^{2,2}_{loc}(x_1' , x_2') $. 
%
Let us prove this claim: Assume that there is $v \in W^{1, \infty} (x_1', x_2') $ such that $v(x_1') = d_1 $ and $v(x_2') = d_2$  and $v(x) \geq \gamma_{0,2} \circ \gamma_{0,1}^{-1} (x)$ such that 
\begin{equation*}
\int_{x_1'}^{x_2'} \sqrt{1+ v'^2} dx < \int_{x_1'}^{x_2'} \sqrt{1+ u'^2} dx. 
\end{equation*}
Then define 
\begin{equation*}
\widetilde{\gamma}(s) := \begin{cases}  \gamma(s) & 0 \leq s \leq t_1 \\ (\gamma_1(s), v( \gamma_1(s)) ) & t_1 \leq s \leq t_2 \\ \gamma(s) & t_2 \leq s \leq 1 
\end{cases}
\end{equation*}
and note that $\widetilde{\gamma} \in M_{\gamma_0} $ because of Proposition \ref{prop:p16}. Using \cite[Theorem 263 D]{Fremlin} the same way it has been used in the proof of Proposition \ref{prop:p16} we find
\begin{eqnarray*}
\mathcal{L}(\widetilde{\gamma}) & = & \int_0^{t_1} | \gamma'| ds + \int_{t_2}^1 |\gamma'| ds + \int_{x_1'}^{x_2'} \sqrt{1+ v'^2} dx 
\\ & <  & \int_0^{t_1} | \gamma'| ds + \int_{t_2}^1 |\gamma'| ds + \int_{x_1'}^{x_2'} \sqrt{1+ u'^2} dx  = \mathcal{L(\gamma)}.
\end{eqnarray*}
However this is a contradiction to the minimizer property of $\gamma$. This completes the proof of the intermediate claim. As already mentioned, the actual regularity follows from Lemma \ref{lem:rgulang}. 
\end{proof}

\begin{cor} \label{cor:regutoppi} (Regularity on the top)

Let $\gamma \in M_{\gamma_0} $ be a minimizer of $\mathcal{L}$ with graph reparametrization $u$. Then $u \in W^{2,2}_{loc}(0,1)$. 
\end{cor}
\begin{proof}
Let $U$ be defined as in Proposition \ref{prop:vartop}. Fix $x \in (0,1)$. If $x \in U$, by virtue of the very same proposition there is an open neighborhood $V$ of $x$ such that $ u \in C^\infty (\overline{V}) \subset W^{2,2}(V)$. Assume that $x \in (0,1)$ is not contained in $U$. As a consequence, there is $t_0\in (0,1)$ such that $\gamma_{0,1}(t_0) = x $ and $u ( \gamma_{0,1}(t_0)) = \gamma_{0,2}(t_0)$ and one of the three possibilities in Lemma \ref{lem:localregu} apply. Assume that the case (1) in Lemma \ref{lem:localregu} applies. 
Then $u_{\mid_{(\delta,x)}}$ is $W^{2,2}_{loc}(\delta,x) $  for each $\delta > 0 $ as graph reparametrization of $(\gamma_{0})_{\mid_{(c,t_0)}}$, for some $c > 0 $. Note that even more holds true: There is $\eta > x $ such that $u_{\mid_{(\delta,\eta)}}$ is $W^{2,2}_{loc}(\delta,\eta) $, since $\gamma_{0,1}'(t_0) > 0 $ and so  the graph reparametrization of $\gamma_{0,1}$ goes a bit further than $x$. In terms of formulas,  for each $\delta' > \delta $ 
\begin{equation*}
\int_{\delta'}^{x} u''^2 dx = \int_{\delta'}^{x} ( \gamma_{0,2} \circ \gamma_{0,1}^{-1} )''^2 dx < \infty
\end{equation*}
because of the pseudograph property.

 So, for some $\delta' \in (\delta, x) $ we find that $ u \in W^{2,2} (\delta' ,x)$.  If now $u$ leaves the coincidence set immediately after $x$ then $(x,x+ \theta)$ is a subset of $U$ for some $\theta > 0 $ and therefore $u \in W^{2,2}(x, x+ \theta)$.
 Since $u \in C^1(0,1)$, the first derivatives match at $x$, and therefore $u$ can be glued to be $W^{2,2} (\delta', x+ \theta)$, which is an open neighborhood of $x$. In case that $x$ is not a boundary point of $(0,1) \setminus U$  we do not just have $u \in W^{2,2} (\delta , x) $ but $ u \in W^{2,2} ( \delta , x + \theta )$ for some $\theta > 0$ since $\gamma_{0,1}'(t_0) > 0 $ and so $u_{\mid_{(\delta, x+ \theta)} } $ is a graph reparametrization of $(\gamma_0)_{\mid_{(c,t_0+ d)}]} $ for some $c, d > 0 $. We are done with case (1). 

If case (2) in Lemma \ref{lem:localregu} applies, the open neigborhood of $V$ of  $x$ such that $u \in W^{2,2}(V)$ can be constructed in the same way. So the only remaining case is $(3)$. But in case (3) it is already part of the statement that there is an open neighborhood $V$ of $x$ such that $u \in W^{2,2}(V) $. So eventually for each $x\in (0,1)$ there exists a neighborhood $V$ of $x$ such that $u \in W^{2,2}(V)$ which results in $u \in W^{2,2}_{loc}(0,1)$. 
\end{proof}

\begin{lemma}(Pseudograph regularity as a curve)\label{lem:regucurve}

Let $\gamma \in M_{\gamma_0}$ be a minimizer of $\mathcal{L}$ in $M_{\gamma_0}$. Then (possibly after weak reparametrization) $\gamma \in P_\psi$ and so $\gamma \in B_{\gamma_0}$. Moreover, $\gamma(\beta_1), \gamma(\beta_2) \in \gamma_0([0,1])$, where $\beta_1, \beta_2$ are as in Definition \ref{defi:panettone}.
\end{lemma}
\begin{proof}
To show the pseudograph property of an element of $M_{\gamma_0} $  we have to show three things: The first one is  $W^{2,2}_{loc}-$regularity of local graph reparametrizations, which however follows right away from Corollary \ref{cor:regutoppi}. The second one is $W^{2,1}_{loc}((0,1);\mathbb{R}^2)$ regularity as a curve and the third one is immersedness. In case that $\beta_1 = 0 $ and $\beta_2 = 1$, there is nothing to show since $\gamma$ starts and ends as a $W^{2,2}_{loc}$-graph which is certainly $W^{2,1}_{loc}$ and immersed. Now suppose that for example $\beta_1 > 0 $, so $\gamma$ is vertical for a while. We have to make sure that $\gamma$ has no "corner" at $\beta_1$, i.e. $\gamma_{\mid_{(\beta_1,\beta_2)}} $ has a vertical tangent vector at $\beta_1$.

Indeed, provided that the graph part has a vertical tangent vector, we can already show $W^{2,1}$-regularity: According to Corollary \ref{cor:conc} and  Proposition \ref{prop:p18} the arclength parametrization of the graph part of $\gamma$ is $W^{2,1}$ and as $\gamma_{[0, \beta_1]}$ is a straight line, the arclength parametrization is certainly $W^{2,1}$.  Since reparametrization does not change the direction of the tangent lines, the tangent vectors of the arclength reparametrizations equal. Glueing componentwise, we obtain that $\gamma$ is $W^{2,1}$ in a neighborhood of $\beta_1$. 
The arguments can be repeated for $\beta_2$ and  the arclength parametrization can be streched to a constant-velocity parametrization without losing regularity. So, the claim really just boils down to showing that the tangent vector is vertical, which we shall do from now on.

First note that $\gamma(\beta_1) \in \gamma_0([0,1]) $ since otherwise there is $\epsilon >0 $ such that $u'' = 0 $ on $(0,\epsilon)$, so $u$ is a line on $(0,\epsilon)$ that does not touch $\gamma_0([0,1])$, so it has positive distance from this curve. Increasing the slope of the line a little bit and shortening the vertical part accordingly would lead to a strictly shorter curve, which is a contradiction to the minimizer property of $\gamma$. Therefore $\gamma(\beta_1) \in \mathrm{tr}(\gamma_0)$. So, there is $t_0 \in (0,1)$ such $\gamma_0(t_0) = \gamma(\beta_1)$. The tangent vector  $\gamma_0'(t_0)$ is vertical since $\gamma_{0,1}'$ is continuous and  $(\gamma_{0,1}')_{\mid_{(0, t_0)}} \equiv 0$. Suppose that $\gamma$ has a non-vertical tangent vector at $\beta_1$. Since $\gamma $ is a graph for $t \in ( \beta_1, \beta_2)$, its arclength parametrization lies in $W^{2,1}$ and since $\gamma_0$ is immersed, so does the arclength parametrization of $\gamma_0$. Therefore, both arclength parametrizations are $C^1$ where $\gamma$ and $\gamma_0$ are graphs. If now the tangent vector of $\gamma$  at $\beta_1$ were not vertical, then $\gamma$ could not lie above $\gamma_0$ in a neighborhood of $\beta_1$. Hence  the tangent vector of $\gamma$ at $\beta_1$ is vertical. 
The same technique can be repeated for $\beta_2$ to obtain the claim finally.    
\end{proof}

\subsection{Existence of a Minimizer for the Trace-Length Problem}

\begin{definition}(Decomposability, \cite[p.1]{Ferriero}) 
Let $N \geq 1$ and $E\subset \mathbb{R}^N$ Lebesgue measurable. We say that $E$ is \emph{decomposable} if there exist measurable sets $A,B \subset \mathbb{R}^N$ such that $E$ is the disjoint union of $A$ and $B$  and $\mathcal{P}(E, \mathbb{R}^N ) = \mathcal{P}(A, \mathbb{R}^N) + \mathcal{P}(B , \mathbb{R}^N) $. 
If $E$ is not decomposable, we call $E$ \emph{indecomposable}.  
\end{definition}

\begin{lemma} (Existence of a minimizer in the class of $\cap$-shaped curves) \label{lem:exmin}

There exists $\gamma \in M_{\gamma_0}$ such that 
\begin{equation*}
\mathcal{L}(\gamma) = \inf_{\iota \in M_{\gamma_0}} \mathcal{L}(\iota) 
\end{equation*}
Additionally, $\gamma \in P_\psi$, $\gamma(\beta_1), \gamma(\beta_2) \in \gamma_0([0,1])$, and $\gamma$ is concave on the top. Additionally, if $u$ denotes the graph reparametrization of $\gamma$, then $u'' = 0 $ a.e. on $\{ x \in (0,1) | u(x) \not \in \gamma_0([0,1])\}$. 
\end{lemma}
\begin{proof} First assume that $\gamma_{0,2} \geq 0.$ 
Fix $\epsilon \in( 0, 1) $. Define
\begin{equation*}
S_x := \sup( \{ y  \in \mathbb{R} | (x,y) \in \gamma_0([0,1]) \}), 
\end{equation*}
\begin{equation*}
E_\epsilon := \bigcup_{ x \in [0,1] } \{x\} \times \left[0 , S_x + \epsilon \right]   \subset \mathbb{R}^2.
\end{equation*}
 Note that the supremum in the definition is actually a maximum due to compactness of $\gamma_0([0,1])$. We claim that
$E_\epsilon$ is Lebesgue measurable in $\mathbb{R}^2$ as a closed subset of $\mathbb{R}^2$. Indeed, if $((x_n,y_n)^T)_{n \in \mathbb{N}}$ is a sequence in $E_\epsilon$ converging to $(x,y)^T \in \mathbb{R}^2$, there is a sequence $(z_n)_{n \in \mathbb{N}}$ such that $ y_n \leq z_n $ and $ (x_n, z_n - \epsilon)^T \in \gamma_0([0,1])$. Compactness of $\gamma_0([0,1])$ implies that there is $z \geq y $ such that $(x,z) \in \gamma_0([0,1])$. The claim follows. From Theorem \cite[Theorem 1]{Ferriero} it follows that
\begin{equation}\label{eq:perimin}
\mathcal{P}(\mathrm{co}(E_\epsilon^1), \mathbb{R}^2) = \inf \{ \mathcal{P}(F, \mathbb{R}^2) : F \supset E_\epsilon \; (\textrm{mod } \lambda_{\mathbb{R}^2}) ,\; F \; \textrm{indecomposable and bounded} \} ,
\end{equation} 
where $\mathcal{P}(\cdot, \mathbb{R}^2)$ denotes the perimeter of a measurable set, $\lambda_{\mathbb{R}^2}$ denotes the $2$-dimensional Lebesgue measure, 
\begin{equation*}
E_\epsilon^1 := \left\lbrace x \in E_\epsilon \; \bigg \vert \; \liminf_{ r\rightarrow 0 } \frac{|E_\epsilon \cap B_r(x)|}{|B_r(x)|} = 1 \right\rbrace ,
\end{equation*}
and $\mathrm{co}(E_\epsilon^1)$ denotes the convex hull of $E_\epsilon^1$. Since $\mathrm{co}(E_\epsilon^1)$ is convex, it has Lipschitz boundary. Therefore, \cite[Section 5.8, Theorem 1]{Evansgariepy} and \cite[Problem 1.5.1]{Toponogov} yield that 
\begin{equation*}
\mathcal{P}(\mathrm{co}(E_\epsilon^1), \mathbb{R}^2) = \mathcal{H}^1(\partial(\mathrm{co}(E_\epsilon^1))) = \mathcal{L}(\widetilde{\gamma})
\end{equation*}
for some convex rectifiable curve $\widetilde{\gamma}$ whose image is $ \partial\mathrm{co}(E_\epsilon^1)$ and $[0,1] \ni s \mapsto \widetilde{\gamma}(s)$ is injective and continuous. Note that $(0,1) \times (0, \epsilon) \subset \overset{\circ}{E_\epsilon} \subset  (E_\epsilon^1)$.
We claim that $\widetilde{\gamma}$ can be chosen such that  $\widetilde{\gamma} = \widetilde{\gamma}^1 \oplus \widetilde{\gamma}^2 \oplus \widetilde{\gamma}^3 \oplus \widetilde{\gamma}^4  $ where $\oplus$ denotes the concatenation of four continuous curves and $\widetilde{\gamma}^1,\widetilde{\gamma}^3$ are vertical lines at $x= 0$ and $x= 1$ respectively, $\widetilde{\gamma}^4$ is a horizontal line at $y = 0$ and $ \widetilde{\gamma}_1^2 \in (0,1) ,\widetilde{\gamma}_2^2 \geq\epsilon$,  and $\widetilde{\gamma}_1^2$ is increasing. Everything except for the monotonicity of $\widetilde{\gamma}_1^2, $ follows from the fact that $(0,1) \times (0, \epsilon) \subset \mathrm{co}(E_{\epsilon}^1)$ and from the fact that $E_\epsilon^1$ is contained in the upper half plane. 
 To show the monotonicity of $\widetilde{\gamma}_1^2$, we show first injectivity, i.e. $\widetilde{\gamma}_1^2(t) = \widetilde{\gamma}_1^2(u)$ implies $t = u$. For this we show that $\widetilde{\gamma}_1^2(t) = \widetilde{\gamma}_1^2(u)$ implies that $\widetilde{\gamma}^2_2(t) = \widetilde{\gamma}^2_2(u)$. Indeed, if $\widetilde{\gamma}^2_2(t) > \widetilde{\gamma}^2_2(u)$, then by convexity of $\mathrm{co}(E_\epsilon^1)$, the triangle  $T$ spanned by $( \frac{1}{2} \widetilde{\gamma}_1^2 , \frac{\epsilon}{2} )^T , \widetilde{\gamma}^2(t) , ( \frac{1}{2}(1 + \widetilde{\gamma}_1^2(t)) , \frac{\epsilon}{2})^T $  is contained in $\mathrm{co}(E_\epsilon^1)$. However, $\widetilde{\gamma}_2^2(u)$ is an interior point of this triangle and hence an interior point of $\mathrm{co} (E_\epsilon^1)$. A contradiction to the fact that $\widetilde{\gamma}^2_2(u) \in \partial \mathrm{co}( E_\epsilon^1)$.  Exchanging roles of $t$ and $u$ one proves that $\widetilde{\gamma}_2^2(t)  = \widetilde{\gamma}_2^2(u)$ and therefore $t=u$. Since $t \mapsto \widetilde{\gamma}_1^2(t)$ is now continuous and injective,  it has to be strictly monotone. Since the direction of parametrization of $\widetilde{\gamma}$ is our choice we can obtain continuous reparametrization of $\widetilde{\gamma}$ with increasing first component in the end. 
Now define $v_\epsilon: (0,1) \rightarrow \mathbb{R}$ by 
\begin{equation*}
v_\epsilon (x) := \sup \{ y \geq 0 | (x,y) \in \overline{\mathrm{co}(E_\epsilon^1)} \}.  
\end{equation*}
We claim that $v_\epsilon \in W^{1,1}(0,1)$. It is easy to show that $v_\epsilon$ is concave and therefore lies in $W^{1,\infty}_{loc}(0,1)$. We show further that $v \in BV(0,1)$. Note first that $v_\epsilon$ is bounded since $\mathrm{co}(E_\epsilon^1)$ is bounded, and thus $v_\epsilon \in L^1(0,1)$. For the rest, we employ \cite[Theorem 1, Section 5.10]{Evansgariepy} and show that $\mathrm{ess}V_0^1 (v_\epsilon)$ is finite. For this let $\mathcal{T}$ be the set of all finite partitions consisting of points of approximate continuity of $v_\epsilon$in $(0,1)$, see \cite[Section 1.7.2]{Evansgariepy}. Then 
\begin{align*}
\mathrm{ess}V_0^1(v_\epsilon) & = \sup_{ (t_i)_{i= 1}^N \in \mathcal{T}, t_1 < t_2 < ... < t_N } \sum_{k = 1}^N |v_\epsilon(t_k)- v_\epsilon(t_{k-1})|  \\ & \leq  \sup_{ (t_i)_{i= 1}^N \in \mathcal{T},t_1 <  t_2 < t_3 < ... < t_N } \sum_{k = 1}^N \sqrt{|t_k - t_{k-1}|^2 + |v_\epsilon(t_k)- v_\epsilon(t_{k-1})|^2} \leq \mathcal{L}(\widetilde{\gamma}^2) ,
\end{align*}
since $(t_k,v_\epsilon(t_k) )_{k =1,...,N}= ( \widetilde{\gamma}^2(u_k) )_{k =1,...,N}$ for some $(u_k)_{k=1,...,N}$ such that $u_1 < u_2 < ... < u_N$, as $u \mapsto \widetilde{\gamma}_1^2(u)$ is increasing. Therefore, by Lemma \ref{lem:l1},  $v_\epsilon \in W^{1,1}(0,1)$. Since $W^{1,1}(0,1) \subset C^0([0,1])$ and $\widetilde{\gamma}([0,1])$ is closed, $v_\epsilon(0), v_\epsilon(1) \in \widetilde{\gamma}([0,1]) $. Using this we find that 
\begin{equation*}
1 + v_\epsilon(0) + v_\epsilon(1) + \int_0^1 \sqrt{1+ v_\epsilon'^2}\; dx \leq \mathcal{L}(\widetilde{\gamma}) = \mathcal{P}(\mathrm{co}(E_\epsilon^1), \mathbb{R}^2) 
\end{equation*}
In particular, $v_\epsilon$ is bounded in $W^{1,1}(0,1)$ since $E_\epsilon \subset [0,1] \times [0, ||\gamma_{0,2}||_\infty + 1]$ for all $\epsilon \in (0,1)$ and therefore, using \eqref{eq:perimin} and \cite[Section 4, Proposition 2]{Ambrosio} one has 
\begin{equation*}
\mathcal{P}(\mathrm{co}(E_\epsilon^1), \mathbb{R}^2) \leq \mathcal{P}((0,1) \times (0, ||\gamma_{0,2}||_\infty + 1), \mathbb{R}^2) \leq 4 + 2 ||\gamma_{0,2}||_\infty. 
\end{equation*}
Due to \cite[Section 5.2.3, Theorem 4]{Evansgariepy} there is a subsequence $\epsilon_n \rightarrow 0 $ and $\widetilde{v} \in BV(0,1)$ such that $v_{\epsilon_n} \rightarrow \widetilde{v}$ in $L^1(0,1)$. Possibly extracting a subsequence, one can also assume that $v_{\epsilon_n} \rightarrow \widetilde{v}$ pointwise almost everywhere. Now define for $x \in (0,1)$, $v(x) := \liminf_{n \rightarrow \infty} v_{\epsilon_n} (x)$. Note that $v = \widetilde{v}$ a.e. and hence $v \in BV(0,1)$. Superadditivity of the Limes Inferior implies that $v$ is concave. Hence, by Lemma \ref{lem:l1}, $v \in W^{1,1}(0,1)$. Define $T := v(0) + v(1) + 1$ and  
\begin{equation*}
\gamma(t) := \begin{cases}
(0, Tt)^T & t \in [0, \frac{v(0)}{T }] \\ ( \; T t- v(0)\; , \; v( T t- v(0)) \;  )^T & t \in (\frac{v(0)}{T}, \frac{v(0) + 1}{T } )  \\ ( 1  , \; v(1) - (Tt- v(0) - 1) \;  )^T & t \in [\frac{v(0) + 1}{T }, 1]
\end{cases} .
\end{equation*}
It is easy to check that $\gamma \in M_{\gamma_0}$ and because of \cite[Theorem 14.2]{Giusti} and superaddititvity of the Limes Inferior
\begin{align}\label{eq:prelang}
1 + \mathcal{L}(\gamma) & = 1+ v(0) + v(1) + \int_0^1 \sqrt{1+ v'^2} dx \leq  \liminf_{n \rightarrow \infty} \left(  1+ v_{\epsilon_n}(0) + v_{\epsilon_n}(1) + \int_0^1 \sqrt{1+ v_{\epsilon_n}'^2} dx \right) \nonumber \\ & \leq \mathcal{P}(\mathrm{co}(E_{\epsilon_n}^1) , \mathbb{R}^2). 
\end{align} 
Now let $\iota \in M_{\gamma_0}$. Then for each $\epsilon > 0 $, let $\widetilde{\iota} $ be the curve that arises from $\iota + \epsilon (0,1)^T$ by gluing with vertical lines that connect $(0,0)^T$ and $(0, \epsilon)^T$ as well as $(1,0)^T$ and $(1, \epsilon)^T$ and a horizontal line that connects $(0,0)^T$ and $(1,0)^T$. By the Jordan curve theorem and \cite[Section 4, Proposition 2]{Ambrosio}, $\widetilde{\iota}$ parametrizes the boundary of an indecomposable set which we will call $F_\epsilon$. Observe that $F_\epsilon \supset E_\epsilon$, as an easy computation shows. Therefore 
\begin{equation*}
\mathcal{L}(\iota) +2 \epsilon + 1 = \mathcal{L}(\widetilde{\iota}) = \mathcal{P}(F_\epsilon, \mathbb{R}^2)  \geq \mathcal{P}(\mathrm{co}(E_{\epsilon}^1) , \mathbb{R}^2) 
\end{equation*}
and using \eqref{eq:prelang} and the last equation with $\epsilon = \epsilon_n$ we obtain
\begin{equation*}
\mathcal{L}(\iota) +2 \epsilon_n + 1 \geq 1 + \mathcal{L}(\gamma).
\end{equation*}
Letting $n \rightarrow \infty $ we find 
\begin{equation*}
\mathcal{L}(\iota) \geq \mathcal{L}(\gamma) .
\end{equation*}
Since $\iota \in M_{\gamma_0}$ was arbitrary, we obtain the claim provided that $\gamma_{0,2} \geq 0 $. For the other direction we first introduce the following notation: For $\iota \in W^{1,1}((0,1);\mathbb{R}^2)$ we define 
\begin{equation*}
\iota^+(t) := (\iota_1(t) , \max\{ 0, \iota_2(t) \} ) \quad (t \in (0,1))  .
\end{equation*}
Observe that $M_{\gamma_0^+} \subset M_{\gamma_0} $ and if $\iota \in M_{\gamma_0} $ then $\iota^+ \in M_{\gamma_0^+}$, with slight abuse of notation since $\gamma_0^+$ does not necessarily lie in $P_\psi$. Additionally, by \cite[Theorem 4 (iii), Section 4.2.2]{Evansgariepy} one has $\mathcal{L}(\iota^+) \leq \mathcal{L}(\iota)$. From the first part of this proof can be concluded (since $W^{2,1}$-regularity of $\gamma_0$ was not needed in this proof so far), that there is $\gamma\in M_{\gamma_0^+}  $ such that 
\begin{equation*}
\mathcal{L}(\gamma) = \inf_{\iota \in M_{\gamma_0^+} } \mathcal{L}(\iota) \leq \inf_{\iota \in M_{\gamma_0}} \mathcal{L}(\iota^+) \leq \inf_{\iota \in M_{\gamma_0}} \mathcal{L}(\iota).  
\end{equation*}
Since $M_{\gamma_0^+} \subset M_{\gamma_0}$, the existence claim follows. The rest of the claim follows from Lemma \ref{lem:regucurve} and Proposition \ref{prop:vartop}.
\end{proof}

\begin{proof}[Proof of Theorem \ref{thm:panettone}]
Let $\gamma_0 \in P_\psi$ be arbitrary. By virtue of Lemma \ref{lem:exmin}, there is $\gamma \in M_{\gamma_0} \cap P_\psi $ concave such that $\mathcal{L}(\gamma) = \inf_{\iota \in M_{\gamma_0}} \mathcal{L}(\iota) $. Let $u$ be its graph reparametrization. We know that $u$ is concave by Corollary \ref{cor:conc}. Now, the straight lines on the side do not contribute at all to the elastic  energy and so we have
\begin{equation*}
\mathcal{E}(\gamma ) = \int_\gamma \kappa^2 ds = \int_0^1  \frac{u''(x)^2}{(1+ u'(x)^2)^\frac{5}{2}} dx. 
\end{equation*}
Additionally, again by Lemma \ref{lem:exmin}, 
\begin{eqnarray*}
\mathcal{L}(\gamma)  & = & \gamma_2(\beta_1) + \gamma_2(\beta_2) + \int_0^1 \sqrt{1+ u'^2}dx   \\ & =& \gamma_{0,2}(l_1) + \gamma_{0,2}(l_2) + \int_0^1 \sqrt{1+ u'^2} dx,
\end{eqnarray*}
where $l_1 = \sup\{ t \in (0,1) |\gamma_{0,1}(t) = 0 \} $ and $l_2 = \inf \{ t \in (0,1) |\gamma_{0,1}(t) = 1 \}$. 
Now denote by $U:= \{x \in (0,1) | (x,u(x)) \not \in \gamma_0([0,1]) \}$ and $I := \{ t \in (0,1) |u \circ \gamma_{0,1} ( t) = \gamma_{0,2} (t) \}$. Note that
\begin{equation*}
u'(\gamma_{0,1}(t)) \gamma_{0,1}'(t) = \gamma_{0,2}'(t) \quad \forall t \in I.  
\end{equation*}
Therefore using \cite[Chapter 2, Lemma A.4]{Kinderlehrerstampacchia} we find that at almost every point of $I$ 
\begin{equation*}
u''(\gamma_{0,1}(t) ) ( \gamma_{0,1}'(t) )^2 + u'( \gamma_{0,1}(t) ) \gamma_{0,1}''(t) = \gamma_{0,2}''(t) .
\end{equation*}
The last identity implies that on $I$ almost everywhere it holds that 
\begin{eqnarray*}
u''(\gamma_{0,1}(t) )^2 & = & \frac{1}{\gamma_{0,1}'(t)^4 } \left( \gamma_{0,2}''(t) - u'(\gamma_{0,1}(t) ) \gamma_{0,1}''(t) \right)^2 \\ & =& 
\frac{1}{\gamma_{0,1}'(t)^4} \left( \gamma_{0,2}''(t) - \frac{\gamma_{0,2}'(t)}{\gamma_{0,1}'(t) } \gamma_{0,1}''(t) \right)^2 \\ & = & 
\frac{1}{\gamma_{0,1}'(t)^6} \left( \gamma_{0,2}''(t) \gamma_{0,1}'(t) - \gamma_{0,2}'(t) \gamma_{0,1}''(t) \right)^2 \\ & =& \frac{1}{\gamma_{0,1}'(t)^6}\left\langle \gamma_0''(t), N(t) \right\rangle^2 |\gamma_0'(t)|^2  .
\end{eqnarray*}
Now observe that by virtue of   the fact that  $u'' \equiv 0 $ on $U$ by Lemma \ref{lem:exmin} and the coarea-formula
\begin{eqnarray*}
\int_0^1 \frac{u''(x)^2}{(1+ u'(x)^2)^\frac{5}{2}} dx &= & \int_{(0,1) \setminus U} \frac{u''(x)^2}{(1+ u'(x)^2)^\frac{5}{2}} \\ &   \leq & \int_{ (0,1) \setminus U } \int_{\{ t \in I : \gamma_{0,1}(t) = x \}} \frac{u''(x)^2}{(1+ u'(x)^2)^\frac{5}{2}} d \mathcal{H}^0(t) dx \\ & =&  \int_I \frac{u''(\gamma_{0,1}(t))^2 }{(1+ u'( \gamma_{0,1}(t)))^\frac{5}{2}} |\gamma_{0,1}'(t)|  dt 
\\ & = & \int_I \frac{1}{\gamma_{0,1}'(t)^6} \left\langle \gamma_0''(t) , N(t) \right\rangle^2 |\gamma_0'(t)|^2  \frac{1}{\left(1 + \left(\frac{\gamma_{0,2}'(t)}{\gamma_{0,1}'(t)}\right)^2 \right)^\frac{5}{2}}|\gamma_{0,1}'(t)| dt \\ & = & \int_I  \frac{\left\langle \gamma_0'' , N \right\rangle^2 }{|\gamma_0'|^3} dt  \leq \int_0^1  \frac{\left\langle \gamma_0'' , N \right\rangle^2 }{|\gamma_0'|^3} dt =  \mathcal{E}(\gamma_0).
\end{eqnarray*}
Note that, since $\gamma_{0,1}$ is (only) locally Lipschitz, we apply here the coarea formula in \cite[Section 3.4.3]{Evansgariepy}  together with the monotone convergence theorem. Additionally, 
\begin{eqnarray}\label{eq:463}
\mathcal{L}(\gamma)  & =  &    \gamma_{0,2}(l_1) + \gamma_{0,2}(l_2) + \int_U \sqrt{1+ u'^2 }dx + \int_{(0,1)\setminus U} \sqrt{1+ u'^2 }dx.
 \end{eqnarray}
The last summand can be simplified similarly to what we just did with $\mathcal{E}$. 
\begin{eqnarray*}
\int_{(0,1) \setminus U } \sqrt{1+ u'(x)^2} dx & \leq & \int_{(0,1) \setminus U } \int_{ \{ t \in I |\gamma_{0,1}(t) =x \} } \sqrt{1+ u'(x)^2} d \mathcal{H}^0(t) dx \\ & =  &
\int_I \sqrt{1 + u'(\gamma_{0,1}(t))^2} |\gamma_{0,1}'(t)| dt = \int_I \sqrt{\gamma_{0,1}'(t)^2 + ( u'(\gamma_{0,1}(t)) \gamma_{0,1}'(t) )^2 }  dt \\ & =& \int_I \sqrt{\gamma_{0,1}'(t)^2 + \gamma_{0,2}'(t)^2 } dt  =  \int_I | \gamma_0'(t) |dt .
\end{eqnarray*}
Now note that $U$ is open and therefore consists of at most countably many connected components $U_i = (x_{2i-1},x_{2i})$ for $i = 1,2,...\;$. For each $j =1,2, ... $ we can find $t_j$ such that $x_j = \gamma_{0,1}(t_j)$. Note that on $(x_{2i-1}, x_{2i})$ it holds that 
\begin{equation*}
u'(x) = \frac{u(x_{2i})- u(x_{2i-1})}{x_{2i}- x_{2i-1}} = \frac{\gamma_{0,2}(t_{2i}) - \gamma_{0,2}(t_{2i-1})}{\gamma_{0,1}(t_{2i}) - \gamma_{0,1}(t_{2i-1})}
\end{equation*}
and therefore, due to the triangle inequality in $\mathbb{R}^2$, 
\begin{eqnarray}\label{eq:kuerzer}
\int_{U} \sqrt{1+ u'(x)^2} dx & =  & \sum_{i = 1}^\infty \int_{x_{2i-1}}^{x_{2i}} \sqrt{1 + \left(  \frac{\gamma_{0,2}(t_{2i}) - \gamma_{0,2}(t_{2i-1})}{\gamma_{0,1}(t_{2i}) - \gamma_{0,1}(t_{2i-1})} \right)^2} dx  \nonumber \\ & =& \sum_{i = 1}^\infty \sqrt{1 + \left(  \frac{\gamma_{0,2}(t_{2i}) - \gamma_{0,2}(t_{2i-1})}{\gamma_{0,1}(t_{2i}) - \gamma_{0,1}(t_{2i-1})} \right)^2}  ( \gamma_{0,1}(t_{2i}) - \gamma_{0,1}(t_{2i-1}) ) \nonumber \\ 
 &= & \sum_{i =1}^\infty || \gamma_0(t_{2i}) - \gamma_0(t_{2i-1}) ||_{\mathbb{R}^2}^2 \nonumber \\ &= & 
\sum_{i = 1}^\infty \left\Vert \int_{t_{2i-1}}^{t_{2i}} \gamma_0'(s) ds \right\Vert_{\mathbb{R}^2}    \leq  \sum_{i=1}^\infty \int_{t_{2i-1}}^{t_{2i}} | \gamma_0'(s) | ds . 
 \end{eqnarray}
 Now note that the intervals $(t_{2i-1}, t_{2i})_{i =1,2,...}$ are disjoint and $\bigcup_{i = 1}^\infty   (t_{2i-1}, t_{2i})  \subset (l_1,l_2) \setminus I$ since if there were a $t' \in I $ such that $t_{2i-1} < t' < t_{2i}$ then  $\gamma_{0,1}(t') \in (0,1) \setminus U$ and $x_{2i-1} \leq \gamma_{0,1}(t') \leq x_{2i}$  which leads to $\gamma_{0,1}(t') \in \{ x_{2i-1} , x_{2i} \} $ and therefore $\gamma_{0,1}$ is constant on either $[t_{2i-1}, t']$ or $[ t', t_{2i} ]$. However $t_{2i}$ and $t_{2i-1}$ are either $l_1, l_2 $ or points of contact. On points of contact, Proposition \ref{prop:c1regu} yields that $\gamma_{0,1}' > 0 $, which contradicts the assertion that $\gamma_{0,1}$ is constant on $[t_{2i-1}, t']$. The remaining possibility $t_{2i-1} = l_1 $ for some $i$  is also impossible since $\gamma_{0,1}(l_1) = 0 $ and $\gamma_{0,1}(t')  = 0 $ would contradict the choice of $l_1$. The same argument applies in the case $l_2 = t_{2i} $. Hence \eqref{eq:kuerzer} leads to the estimate 
 \begin{equation*}
 \int_U \sqrt{1+ u'^2} dx \leq \int_{(0,1) \setminus I} | \gamma_0'(s) |ds. 
 \end{equation*}
 Eventually, plugging the estimates into \eqref{eq:463} we end up with
 \begin{equation*}
 \mathcal{L}(\gamma) \leq \gamma_{0,1}(l_1) + \gamma_{0,2}(l_2) + \int_{l_1}^{l_2} |\gamma_0'| dt = \mathcal{L}(\gamma_0),
 \end{equation*}
 from which follows that $\mathcal{L}(\gamma) \leq \mathcal{L}(\gamma_0)$ and together with $\mathcal{E}(\gamma) \leq \mathcal{E}(\gamma_0)$ it also follows that $\mathcal{E}_\epsilon(\gamma) \leq \mathcal{E}_\epsilon(\gamma_0)$ for nonnegative $\epsilon$. 
\end{proof}

\section{Existence via Penalization}

Combining Proposition \ref{prop:p8} and Theorem \ref{thm:panettone}, we obtain that $\mathcal{E}_\epsilon$ admits a minimizer $\gamma_\epsilon$, which is a $\cap$-shaped pseudograph and concave on the top. Note that $\mathcal{E}_\epsilon(\gamma_\epsilon)< \infty$ means that the constant length reparametrization of $\gamma_\epsilon$ is an element of $W^{2,2}((0,1);\mathbb{R}^2)$. If we can bound the length of minimizers uniformly in $\epsilon$ as well as the energies, the $L^2$-norm of the first and second derivative of $(\gamma_\epsilon)$ will be bounded uniformly in $\epsilon$ (see Proposition \ref{prop:p7}). Using this, we can extract a subsequence that converges weakly in $W^{2,2}$ to a minimizer. 

A uniform bound of the energy is easy to derive using a certain monotonicity. A uniform bound for the length however, can only be expected in case that $\inf_{\gamma \in P_{\psi}} \mathcal{E}(\gamma) < c_0^2$, see Figure \ref{fig:3}. This is the reason for the case distinction in Theorem \ref{thm:main}.

\begin{lemma}\label{lem:lengthbd} (A length bound for one-sided graphs) 

Let $\gamma\in P_\psi$ be a $\cap$-shaped pseudograph that is concave on the top. Let $u$ be the graph associated to $\gamma$. If $\beta_1 = 0 $ then 
\begin{equation*}
\mathcal{L}(\gamma) \leq  u'(0) + \sqrt{1+ u'(0)^2}.
\end{equation*}
If $\beta_2 = 1 $ then 
\begin{equation*}
\mathcal{L}(\gamma) \leq |u'(1)| + \sqrt{1+ u'(1)^2}. 
\end{equation*}
\end{lemma}
\begin{proof}
It suffices to show the first part since we can consider $\gamma(1- \cdot) $ otherwise. 
Consider the triangle $T$ with vertices $(0,0), (1,0) $ and $(1, u'(0))$. As a triangle, $T$ is convex and $T$ can be expressed by 
\begin{equation*}
T = \bigcup_{x \in [0,1]} \{x \} \times [0, u'(0) x ] . 
\end{equation*}

 Define 
\begin{equation*}
E := \bigcup_{x \in [0,1] } \{ x \} \times [0, u(x)]. 
\end{equation*}
Certainly, $E$ is convex since $u$ is concave. Note that $\gamma$ is a parametrization of $\partial E$ with the bottom removed. We claim that 
\begin{equation*}
 \mathcal{H}^1(\partial E) \leq \mathcal{H}^1(\partial T)
\end{equation*}
which results in 
\begin{equation*}
1 + \mathcal{L}(\gamma) \leq 1 + u'(0) + \sqrt{1+ u'(0)^2}, 
\end{equation*}
that implies the statement. 
We will prove that $E \subset T$ and $P_E : \partial T \rightarrow \partial E$ is surjective and contractive, where $P_E$ denotes the best approximation map. The relation $E \subset T$ follows from concavity, since
\begin{equation*}
u(x) = u(x) - u(0) \leq u'(0) (x- 0) \leq u'(0) x. 
\end{equation*} 
Observe that on the sides and on the bottom of $E$ it holds that $P_E = \mathrm{Id} $ and therefore the sides and bottom lie in $P_E(\partial T)$. Assume now that there is $z \in \partial E$ on the top such that $z \not \in P_E(\partial T)$. Now $u$ has a tangent line $L$ at $z$ and $E$, as a convex set, lies only on one side of $L$. We prolong the outer unit normal of $E$ at $z$ until we hit $\partial T$ to find that there is $p \in \partial T$ such that $ p -z  $ is perpendicular to $L$. Now however $P_E(p) \neq z$. Since $L$ seperates $p$ and $E$ we can find a convex combination of $p$ and $P_E(p)$, say $k = t p + (1-t)P_E(p) $ for some $t \in [0,1]$, that lies in $L$. Note that 
\begin{equation*}
||k - p|| = (1-t) || p- P_E(p) || < || p -z||. 
\end{equation*}
However, since $z \in L$  and $p -z \perp L $
 \begin{equation*}
 ||p -k||^2 = ||p-z||^2 + ||z-k||^2 \geq || p -z ||^2  ,
 \end{equation*}
 which is a contradiction. So, $P_E: \partial T \rightarrow \partial E $ is surjective and contractive by \cite[Proposition 5.3]{Brezis}. Therefore
 \begin{equation*}
 \mathcal{H}^1(\partial T ) \geq \mathcal{H}^1 ( P_E(\partial T ) ) = \mathcal{H}^1(\partial E) ,
 \end{equation*}
 which proves the claim. 
\end{proof}
\begin{lemma}\label{lem:lenghtbound2} (A length bound for functions that touch the obstacle)
 
Let $\gamma \in P_\psi$ be a $\cap$-shaped pseudograph that is concave on the top. If $\gamma$ touches the obstacle $\psi$ then
\begin{equation*}
\mathcal{L}(\gamma) \leq 2 \left( \sup_{x \in (0,1)} \psi(x) + || \psi'||_\infty \right)  + 1. 
\end{equation*} 
\end{lemma}
\begin{proof}
Let $u$ be the graph reparametrization of $\gamma$. 
We first show that $||u||_\infty \leq \sup_{x \in (0,1)} \psi(x) + ||\psi'||_\infty $. Note that if $u(x_0) = \psi(x_0) $, then $u \in C^1(0,1)$ implies that
\begin{equation*}
\partial_{+} \psi'(x_0) \leq u'(x_0) \leq \partial_{-} \psi(x_0) 
\end{equation*}
and therefore 
\begin{equation*}
|u'(x_0)| \leq \max ( |\partial_{+} \psi(x_0)|, | \partial_{-} \psi(x_0) | ) \leq ||\psi'||_\infty.
\end{equation*}
Now note that for each $x \in (0,1)$ 
\begin{equation*}
u(x) \leq u(x_0) + u'(x_0)(x-x_0) \leq \psi(x_0) + || \psi' ||_\infty |x- x_0|  ,
\end{equation*}
and therefore $|| u||_\infty \leq \sup_{x \in (0,1)} \psi(x) + || \psi' ||_\infty $. For the rest of the proof let $S := \sup_{x \in (0,1) } \psi(x) $ and define $R$ to be the rectangle with vertices $(0,0)^T, (0, S + ||\psi'||_\infty)^T , (1, S + ||\psi'||_\infty )^T$ and $(1,0)^T$. Adapting the techniques from Lemma \ref{lem:lengthbd} we find that 
\begin{equation*}
1 + \mathcal{L}(\gamma) \leq \mathcal{H}^1 ( \partial R ) = 2 (S + || \psi'||_\infty ) + 2 
\end{equation*}
This proves the claim. 
\end{proof}

\begin{lemma}\label{lem:bigenergy}
Let $c_0$, $G$ and $U_0, S$ be defined as in Theorem \ref{thm:main}. Additionally, let $\gamma= \gamma_1 \oplus \gamma_2 \oplus \gamma_3$ be as in \eqref{eq:vglkurve}. Then $\gamma \in P_\psi$, 
\begin{equation*}
\mathcal{L}(\gamma) = 2 S + \frac{1}{c_0} \int_\mathbb{R}\frac{1}{(1+t^2)^\frac{3}{4}} dt ,
\end{equation*}
and 
\begin{equation*}
\mathcal{E}(\gamma) = c_0^2 .
\end{equation*}
\end{lemma}
\begin{proof}
Note that 
\begin{eqnarray*}
U_0'(x) & = & \frac{1}{(  1 + G^{-1} ( \frac{c_0}{2} - c_0 x)^2 )^\frac{5}{4}}  G^{-1} \left( \frac{c_0}{2} - c_0 x \right) \frac{1}{G'\left( G^{-1} \left( \frac{c_0}{2} - c_0 x \right) \right) } \\ &= & G^{-1} \left(\frac{c_0}{2} - c_0 x \right) 
\end{eqnarray*}
and 
\begin{equation*}
U_0''(x)  = - c_0 \left( 1+ G^{-1} \left( \frac{c_0}{2} - c_0 x \right)^2 \right)^\frac{5}{4}.
\end{equation*}
Therefore, $U_0+ S$, the graph reparametrization of $\gamma$  lies certainly in $W^{2,2}_{loc}( (0,1);\mathbb{R}^2)$. For the regularity, i.e. $\gamma \in P_\psi $, we can proceed like in Lemma \ref{lem:regucurve}: Take arclength parametrizations of all three pieces, observe that all the pieces are $W^{2,1}$ using Proposition \ref{prop:p18}, and have matching tangent vector at the points where they meet. 
Since
\begin{equation*}
\frac{U_0''(x)^2}{(1 + U_0'(x)^2)^\frac{5}{2} } = c_0^2 \quad \forall x \in (0,1), 
\end{equation*}
 we obtain $\mathcal{E}(\gamma) = c_0^2$. Now,
\begin{align*}
\mathcal{L}(\gamma)  & =   2 S + \int_0^1 \sqrt{1 + U_0'(x)^2 } dx  = 2S+ \frac{1}{c_0}\int_{-\frac{c_0}{2}}^{\frac{c_0}{2}} \sqrt{1 + G^{-1} (s)^2 } ds \\ &  = 2 S  + \frac{1}{c_0} \int_{- \infty}^{\infty} \sqrt{1 + G^{-1} (G(t))^2} G'(t) dt  = 2 S + \frac{1}{c_0} \int_{\mathbb{R}} \frac{\sqrt{1+t^2}}{(1+ t^2)^\frac{5}{4}} dt .
\end{align*}
This implies the claim. 
\end{proof}

\begin{proof}[Proof of Theorem \ref{thm:main}]
Claim (1), existence in case that $\alpha = c_0^2$ and claim (3) follow directly from Lemma \ref{lem:bigenergy}.
It remains to consider the case $\alpha < c_0^2$. So, suppose that there is $\gamma \in P_\psi$ such that $\mathcal{E}(\gamma) < c_0^2 $. Choose $\delta:= \frac{c_0^2 - \alpha }{2 }$. There exists $\epsilon_0 > 0 $  such that $\mathcal{E}_\epsilon( \gamma) \leq c_0^2 - \delta $ for each $\epsilon<  \epsilon_0$. Choose $\gamma_\epsilon$ to be a  $\cap$-shaped minimizer of $\mathcal{E}_\epsilon$ that is concave on the top, which - recall - can be constructed using Proposition \ref{prop:p8} and Theorem \ref{thm:panettone}. Therefore 
\begin{equation*}
\mathcal{E}_\epsilon(\gamma_\epsilon) \leq \mathcal{E}_\epsilon(\gamma) \leq c_0^2 - \delta
\end{equation*}
for each $\epsilon < \epsilon_0$. We claim that for each such $\epsilon$ either $\beta_1(\gamma_\epsilon) = 0 $ and/or $\beta_2(\gamma_\epsilon) = 0 $, see \eqref{eq:panettone} fo the definition of $\beta_1, \beta_2$. If both are nonzero and $u_\epsilon$ denotes the graph reparametrization of $\gamma_\epsilon$, then $u_\epsilon'(0) = \infty $, $u_\epsilon'(1) = - \infty $ (by Lemma \ref{lem:regucurve}) and therefore by virtue of Proposition \ref{prop:p1} (or more precisely a very small variation of it) 
\begin{equation*}
\mathcal{E}_\epsilon(\gamma_\epsilon) \geq \mathcal{E}(\gamma_\epsilon) \geq (G(u_\epsilon'(0)) - G(u_\epsilon'(1))^2 = c_0^2, 
\end{equation*}
a contradiction. From this computation also follows that $u_\epsilon'(0), u_\epsilon'(1)$ cannot both be infinite. Without loss of generality we can assume that for infinitely many $\epsilon > 0 $ we have $\beta_1(\gamma_\epsilon) = 0 $ and $u_\epsilon'(0) < \infty $. From now on consider only such $\epsilon$. If $u_\epsilon'(0), |u_\epsilon'(1)| < \infty $ it follows from Proposition \ref{prop:p2} that $u_\epsilon \in W^{2,2}(0,1)$ and therefore $u_\epsilon$ is a graph that touches the obstacle, due to Proposition \ref{prop:bpunkt}. Hence, according to Lemma \ref{lem:lenghtbound2},
\begin{equation*}
\mathcal{L}(\gamma_\epsilon) \leq 2 \left( \sup_{x \in (0,1)} \psi(x)  + || \psi'||_\infty \right) + 1. 
\end{equation*}
In case that $u_\epsilon'(0) < \infty $, $u_\epsilon'(1)= -\infty $, recall that again by Proposition \ref{prop:p1}
\begin{equation*}
c_0^2 -  \delta \geq \mathcal{E}_\epsilon(\gamma_\epsilon) \geq (G (u_\epsilon'(0))- G(u_\epsilon'(1)))^2 = \left( G(u_\epsilon'(0)) + \frac{c_0}{2} \right)^2. 
\end{equation*}
Therefore 
\begin{equation*}
u_\epsilon'(0) \leq G^{-1} \left(  \sqrt{c_0^2 - \delta }- \frac{c_0}{2} \right)
\end{equation*}
and hence Lemma \ref{lem:lengthbd} yields that 
\begin{equation*}
\mathcal{L}(\gamma_\epsilon) \leq G^{-1} \left(  \sqrt{c_0^2 - \delta }- \frac{c_0}{2} \right) + \sqrt{1+ G^{-1} \left( \left(  \sqrt{c_0^2 - \delta }- \frac{c_0}{2} \right) \right)^2 }. 
\end{equation*}
All in all, we can infer that $\mathcal{L}(\gamma_\epsilon) \leq C$, which is exactly the bound given in Theorem \ref{thm:main} by our choice of $\delta$. Now consider the constant-velocity reparametrization of $\gamma_\epsilon$ which we will  call $\gamma_\epsilon$ again. Then, as an obvious consequence of Proposition \ref{prop:p7}, $\gamma_\epsilon$ is bounded in $W^{2,2} ((0,1);\mathbb{R}^2) $ and therefore possesses a $W^{2,2}-$weakly convergent subsequence $\gamma_{\epsilon'} $ converging to some $\gamma \in W^{2,2}((0,1);\mathbb{R}^2)$, satisfying 
\begin{equation*}
\mathcal{E}(\gamma ) \leq \liminf_{\epsilon' \rightarrow 0} \mathcal{E} (\gamma_{\epsilon'}) . 
\end{equation*}
Note that $\gamma$ also satisfies 
\begin{equation*}
\mathcal{L}(\gamma) = \lim_{ \epsilon' \rightarrow 0 } \mathcal{L}(\gamma_{\epsilon'}) \leq C.
\end{equation*}
Additionally, $\gamma$ is parametrized with constant velocity as weak $W^{2,2}$ limit of curves that are parametrized with constant velocity. It remains to show that $\gamma \in P_\psi$ but this is very similar to one of the crucial steps in the proof of Proposition \ref{prop:p8}. 
Further,  $\gamma$ is indeed a minimizer of $\mathcal{E}$ since for fixed $\iota \in P_\psi$ we find 
\begin{align*}
\mathcal{E}(\iota) & = \lim_{\epsilon \rightarrow 0 } \mathcal{E}_\epsilon(\iota) 
\geq   \limsup_{\epsilon \rightarrow 0 } \mathcal{E}_\epsilon(\gamma_\epsilon) =  \limsup_{\epsilon \rightarrow 0 }  (\mathcal{E}(\gamma_\epsilon) + \epsilon \mathcal{L}(\gamma_\epsilon)) \\ & \geq  \liminf_{\epsilon' \rightarrow 0 } (\mathcal{E}(\gamma_{\epsilon'}) + \epsilon' \mathcal{L}(\gamma_{\epsilon'}))=  \liminf_{\epsilon' \rightarrow 0 } \mathcal{E}(\gamma_{\epsilon'} )  \geq \mathcal{E}(\gamma), 
\end{align*}
 having used the weak lower semicontinuity of $\mathcal{E}$ and the fact that $\mathcal{L}(\gamma_{\epsilon'})$ can be bounded uniformly in $\epsilon'$. 

\end{proof}
\begin{remark}
We have found a framework for the obstacle problem in \cite{Anna}, where a minimizer  exists. However, \cite{Anna} provides a lot more results on the shape of a minimizer, which brings up more questions in the new framework. One example is the question, whether symmetric obstacles necessarily lead to existence of symmetric minimizers. Having characterized a minimizer of $\mathcal{E}$ as a limit of $\mathcal{E}_\epsilon$ and knowing some properties of the (Jacobi-elliptic) functions that solve the Euler-Lagrange equation for $\mathcal{E}_\epsilon$, for example from \cite[Section 4.3]{Miura3} and from \cite{Mandel}, it might be possible to understand minimizers more explicitly than we do so far. 
\end{remark}
\begin{remark}
Let us stress the relation of our approach to the theory of $\Gamma$-convergence. For $j \in \mathbb{N}$ consider the functional $F_j : W^{1,2}((0,1);\mathbb{R}^2) \rightarrow \overline{\mathbb{R}}$ 
\begin{equation*}
F_j(\gamma) := \begin{cases} 
\int_\gamma \kappa^2 d\mathbf{s} + \frac{1}{j} \int_\gamma d\mathbf{s} & \gamma \in W^{2,2}((0,1);\mathbb{R}^2) \cap P_\psi  \, \textrm{parametrized s.t. $|\gamma'| \equiv$ const.} \\ \infty  & \textrm{otherwise} 
\\ \end{cases}. 
\end{equation*}
Then $(F_j)_{j \in \mathbb{N}}$ is a decreasing sequence and as such, it $\Gamma$-converges according to \cite[Remark 1.40]{Braides} to the lower semi-continuous envelope of the pointwise limit of $(F_j)_{j \in \mathbb{N}}$. Obviously, the pointwise limit is given by 
\begin{equation*}
F(\gamma)  = \begin{cases} \int_\gamma \kappa^2 d\mathbf{s} & \gamma \in W^{2,2}((0,1);\mathbb{R}^2) \cap P_\psi \;  \textrm{and is parametrized with constant velocity}  \\ \infty & \textrm{otherwise} \end{cases}
\end{equation*} 
Let us show that $F$ is lower semi-continuous in $W^{1,2}((0,1);\mathbb{R}^2)$, so that its lower semi-continuous envelope coincides with $F$. Let $(\gamma_n)_{n \in \mathbb{N}}$ be a sequence that converges to some $\gamma$ in $W^{1,2}$ and satisfies that $(F(\gamma_n))_{n \in \mathbb{N}}$ is bounded. Then $(\gamma_n)_{n \in \mathbb{N}}$ defines a bounded sequence in $W^{2,2}((0,1);\mathbb{R}^2)$ and has a weakly convergent subsequence in $W^{2,2}((0,1);\mathbb{R}^2)$. Using the embedding $W^{2,2} \hookrightarrow C^1$ it can be shown that $\gamma$ is parametrized with constant velocity and as such particularly immersed. We infer that $\gamma \in P_\psi$ using the techniques of the proof of Proposition \ref{prop:p8}. Since $\mathcal{E}$ is weakly lower semi-continuous we obtain that 
\begin{equation*}
F(\gamma) \leq \liminf_{n \rightarrow \infty } F(\gamma_n). 
\end{equation*} 
We assumed to begin with that $(F( \gamma_n ))_{n \in \mathbb{N}}$ is bounded, but this is not restrictive, since as long as $\liminf_{n\rightarrow \infty} F(\gamma_n)$ is a real number, we can pick a bounded subsequence converging to the Limes Inferior. All in all we have shown that $F$ is the $\Gamma$-limit of $(F_j)_{j \in \mathbb{N}}$ in $W^{1,2}((0,1);\mathbb{R}^2)$. 
\\

Recall now the fundamental theorem of $\Gamma$-convergence  \cite[Theorem 1.21]{Braides}, saying that provided that we can find a compact set $ K \subset W^{1,2}((0,1);\mathbb{R}^2)$ such that 
\begin{equation*}
\forall j \in \mathbb{N} : \qquad  \inf_{W^{1,2}((0,1);\mathbb{R}^2)} F_j  = \inf_K F_j 
\end{equation*}  
then there exists a minimizer for $F$ and every precompact sequence $(\tau_j)$  such that 
\begin{equation}\label{eq:approxim}
F_j(\tau_j) = \inf_{W^{1,2}} F_j + o(1) 
\end{equation}
has a subsequence that converges to some minimizer of $F$. Especially
\begin{equation*}
\min_{W^{1,2}} F = \lim_{j \rightarrow \infty } \inf_{W^{1,2}} F_j.
\end{equation*}
 Let $\gamma_j$ be the sequence of constant-velocity-parametrized minimizers of $\mathcal{E}_\frac{1}{j}$ as already considered in the proof of Theorem \ref{thm:main}. With the same arguments $\{ \gamma_j | j \in \mathbb{N} \}$ is bounded in $W^{2,2}$ so precompact in $W^{1,2}$ which means that its closure is compact. Therefore, $K$ can be chosen to be this exact closure to ensure that the prerequisites of the Fundamental Theorem of $\Gamma$-convergence are satisfied. We actually get an even more interesting result out of this: 
 \\
 
Each minimizer $\gamma$ of $F$ is a $W^{1,2}$- limit of a sequence $(\tau_j)$ satisfying \eqref{eq:approxim}. Indeed: If $\gamma$ is a minimizer of $F$, then there exists a recovery sequence for $\gamma$, which would satisfy \eqref{eq:approxim} and the claim follows. 
\end{remark}
\appendix

 \section{Proofs of Results in Section 2.1}
\begin{proof}[Proof of Proposition \ref{prop:p1}] 
Using the Cauchy-Schwarz inequality we find
\begin{eqnarray*}
\mathcal{E}(u) & = & \int_0^1 \frac{u''(x)^2 }{(1 + u'(x)^2)^\frac{5}{2}}   \geq  \left( \int_0^1 \frac{|u''(x)|}{(1+ u'(x)^2)^\frac{5}{4}} \right) \\ &  \geq & \left( \int_{b_1}^{b_2} \frac{|u''(x)|}{(1 + u'(x)^2)^\frac{5}{4}} dx \right)  \geq   \left( \int_{b_1}^{b_2} \frac{d}{dx} (G \circ u' )(x) dx \right)^2
\end{eqnarray*}
since $G \in C^\infty(\mathbb{R}), G(0) = 0 $ and $G$ has bounded derivative, so  \cite[Theorem 4(ii), Section 4.2.2.]{Evansgariepy} applies. 
\end{proof}

\begin{proof}[Proof of Proposition \ref{prop:p5}] The $C^2$-regularity follows from \cite[Theorem 5.1]{Anna} and concavity is a direct consequence of \cite[Lemma 2.1]{Anna}.
Assertion $(2)$ follows from  \cite[Proposition 3.2]{Anna} using the product rule and that 
\begin{equation*}
\frac{d}{dx} \frac{1}{(1+ u'(x)^2)^\frac{3}{4}} = -\frac{3}{2}\frac{u''(x) u'(x)}{(1+ u'(x)^2)^\frac{7}{4}} = - \frac{3}{2}\frac{k(x) u'(x)}{(1+ u'(x)^2)^\frac{1}{4}}.
\end{equation*}
For (3) we compute 
\begin{eqnarray*}
\frac{d}{dx} \frac{v'(x)}{(1 + u'(x)^2)^\frac{5}{4} } &=& \frac{v''(x)}{(1+ u'(x)^2)^\frac{5}{4}}  - \frac{5}{2} \frac{v'(x) u'(x) u''(x)}{(1+ u'(x)^2)^\frac{9}{4}} \\ & =&  \frac{v''(x)}{(1+ u'(x)^2)^\frac{5}{4}}  - \frac{5}{2} \frac{v'(x) u'(x) k(x)}{(1+ u'(x)^2)^\frac{3}{4}} \\ &= & \frac{1}{\sqrt{1+ u'(x)^2}} \left( \frac{v''(x)}{(1+ u'(x)^2)^\frac{3}{4}}  - \frac{5}{2} \frac{v'(x) u'(x) k(x)}{(1+ u'(x)^2)^\frac{1}{4}} \right) = 0. 
\end{eqnarray*}
Assertion $(4)$ follows from \cite[Corollary 3.3]{Anna} and (5) follows from (2) by the maximum principle (\cite[Section 6.4, Theorem 2]{Evans}). For $(6)$ observe that there is $\delta > 0 $ such that $u > \psi$ on $[0, \delta]$ and according to assertion (1), since $\mathrm{sgn}(u'') = \mathrm{sgn}(v)$, it holds that $v \leq 0 $. Hence it holds for $0 \leq a \leq b \leq \delta$
\begin{equation*}
-v(a) = |v(a) | = \max_{x \in [0,a]} |v(x)| \leq \max_{x \in [0,b]} |v(x)| = \max\{ |v(b)|, |v(0)|\}   = |v(b)| = -v(b) . 
\end{equation*}
Multiplying by $-1$ we have shown the decrease in a neighborhood of $0$. The proof in analogous for a neighborhood of $1$. Assertion (7) is \cite[Corollary 3.4]{Anna}. 
\end{proof}

\begin{proof}[Proof of Proposition \ref{prop:p2}]
Since $u' \in L^\infty(0,1)$, we find 
\begin{equation*}
\mathcal{E}_\epsilon (u) = \int_0^1 \left( \frac{u''^2}{(1+ u'^2)^\frac{5}{2}} dx + \epsilon  \sqrt{1+ u'^2} \right)  dx \geq \frac{1}{( 1 + ||u'||_\infty^2)^\frac{5}{2}} \int_0^1 u''^2 dx .
\end{equation*}
Therefore $\mathcal{E}_\epsilon(u) < \infty $ implies that $||u''||_{L^2(0,1)} < \infty$.
\end{proof}

\section{Proofs of Results in Section 2.2}
\begin{proof}[Proof of Proposition \ref{prop:p7}]
If $\gamma$ is parametrized with constant velocity, then $|\gamma'| \equiv \mathcal{L}(\gamma)$. Let $ T = \frac{\gamma'}{|\gamma'|} $ be the tangential vector. Then $\{ T , N \} $ forms an orthonormal basis of $\mathbb{R}^2$ and thus
\begin{eqnarray*}
\mathcal{E}(\gamma) & =&  \int_0^1 \frac{\left\langle \gamma'', N \right\rangle^2}{|\gamma'|^3} dt \\ & =& \int_0^1 \frac{|\gamma''|^2- \left\langle \gamma'', T \right\rangle^2}{\mathcal{L}(\gamma)^3} dt.
\end{eqnarray*}
Now note that 
\begin{equation*}
0 = \frac{d}{dt} \mathcal{L}(\gamma)^2 = \frac{d}{dt} |\gamma'|^2 = 2 \left\langle \gamma'', \gamma' \right\rangle, 
\end{equation*} 
so $\left\langle \gamma'', T \right\rangle = 0 $ and
\begin{equation*}
\mathcal{E}_\epsilon(\gamma) = \frac{1}{\mathcal{L}(\gamma)^3}\int_0^1 | \gamma''|^2 + \epsilon \mathcal{L}(\gamma).  
\end{equation*}
  Note also that $\mathcal{L}(\gamma) < \infty $ and therefore $\mathcal{E}_\epsilon(\gamma) < \infty $ yields that $ ||\gamma''||_{L^2}^2 < \infty.$ 
\end{proof}

\begin{proof}[Proof of Proposition \ref{prop:p8} ] 
Fix $\epsilon > 0 $ and assume that $(\gamma_n) \subset P_\psi$ is a minimizing sequence for $\mathcal{E}_\epsilon$ parametrized with constant velocity. According to Proposition \ref{prop:p7} $(\gamma_n)_{n \in \mathbb{N}} \subset W^{2,2} ((0,1);\mathbb{R}^2)$ and satisfies 
\begin{equation}\label{eq:constivelo}
\mathcal{E}_\epsilon(\gamma_n) = \frac{1}{\mathcal{L}(\gamma_n)^3} \int_0^1 |\gamma_n''|^2 dx + \epsilon \mathcal{L}(\gamma_n).
\end{equation}
Let $M > 0 $ be such that $\mathcal{E}_\epsilon(\gamma_n) \leq M$ for all $n \in \mathbb{N}$. We infer from \eqref{eq:constivelo}  first that 
\begin{equation*}
\mathcal{L}(\gamma_n) \leq \frac{M}{\epsilon}, 
\end{equation*}
and then
\begin{equation*}
\int_0^1 |\gamma_n''|^2 \leq \frac{M^4}{\epsilon^3}.
\end{equation*}
This together with the boundary conditions yields that $(\gamma_n)_{n \in \mathbb{N}}$ defines a bounded sequence in $W^{2,2}((0,1);\mathbb{R}^2) $ and therefore we can extract a weakly $W^{2,2}((0,1);\mathbb{R}^2)$-convergent subsequence, the weak limit of which we will call $\gamma$. Weak lower semicontinity of the $L^2$-norm and the fact the $\mathcal{L}(\gamma_n) \rightarrow \mathcal{L}(\gamma)$ by virtue of compactness of the embedding $W^{2,2}((0,1);\mathbb{R}^2) \hookrightarrow W^{1,2}((0,1);\mathbb{R}^2)$ implies that 
\begin{equation*}
\mathcal{E}_\epsilon(\gamma)  \leq \liminf_{n \rightarrow \infty } \mathcal{E}_\epsilon(\gamma_n). 
\end{equation*} 
It remains to show that $\gamma$ is a pseudograph. For this, we prove now that $\gamma$ is immersed, attains the prescribed values at the boundary, and every local graph reparametrization of $\gamma$ lies in $W^{2,2}_{loc}$. That $\gamma
$ attains the right values at the boundary follows from continuity of the embedding $W^{1,2}((0,1);\mathbb{R}^2) \hookrightarrow C( [0,1]; \mathbb{R}^2 )$.
That $\gamma$ is immersed follows from  the compact embedding $W^{2,2} \hookrightarrow C^1 $ and 
\begin{equation*}
|\gamma'(t)| = \lim_{n \rightarrow \infty} |\gamma_n'(t)| = \lim_{n \rightarrow \infty } \mathcal{L}(\gamma_n).
\end{equation*}
as well as 
\begin{equation*}
\liminf_{n \rightarrow \infty} \mathcal{L}(\gamma_n) \geq  \liminf_{n \rightarrow \infty} || \gamma_n(1) - \gamma_n(0) || = 1 . 
\end{equation*}
Assume now that $\gamma_1' > 0 $ on some interval $(a,b)$. Again, continuity of the embedding $W^{2,2}(a,b) \hookrightarrow C^1([a,b])$  implies that for each $\theta > 0$ there is $\delta > 0 $ such that $\frac{1}{\delta} > \gamma_1' > \delta $ on $(a+ \theta , b- \theta) $. Therefore $\gamma_{2} \circ \gamma_1^{-1}$ is a $W^{2,2}$ function composed with a $C^1$-diffeomorphism in $(a+ \theta, b- \theta)$. Certainly we can compute the first derivative classically : 
\begin{equation*}
\frac{d}{dx} (\gamma_2 \circ \gamma_1^{-1})(x) = \gamma_2'( \gamma_1^{-1} (x) ) (\gamma_1^{-1})'(x) = \frac{\gamma_2'(\gamma_1^{-1} (x)) }{\gamma_1'(\gamma_1^{-1} (x)) } .
\end{equation*} 
For the computation to come note that an elementary computation using \cite[Section 4.2.2, Theorem 4]{Evansgariepy} would reveal that $\frac{\gamma_2'}{\gamma_1'}$ is weakly differentiable in $(a+ \theta , b- \theta)$ with weak derivative given by $ \frac{\gamma_2'' \gamma_1' - \gamma_1'' \gamma_2'}{(\gamma_1')^2}$. 
For the second derivative define $d := \gamma_1(a+\theta), e:= \gamma_1(b-\theta)$ and fix $\phi\in C_0^\infty(d,e)$. Then, using \cite[Section 3.4.3, Theorem 2]{Evansgariepy} and recalling that $\phi \circ \gamma_1 \in C^1_0(a+ \theta,b- \theta)$ we find 
\begin{align*}
\int_{d}^{e} \frac{\gamma_2'(\gamma_1^{-1}(x))}{\gamma_1'(\gamma_1^{-1}(x)) } \phi'(x) dx & = \int_{a+ \theta}^{b- \theta} \frac{\gamma_2'(t)}{\gamma_1'(t)} \phi'(\gamma_1(t)) \gamma_1'(t) dt \\ &=  \int_{a+ \theta}^{b- \theta} \frac{\gamma_2'(t) }{\gamma_1'(t) } \frac{d}{dt} \left[ \phi(\gamma_1(t) ) \right] dt =  - \int_{a+ \theta}^{b- \theta} \left( \frac{d}{dt} \frac{\gamma_2'(t)}{\gamma_1'(t) } \right) \phi(\gamma_1(t) ) dt \\ &=  - \int_{a+ \theta}^{b- \theta} \frac{\gamma_2''(t) \gamma_1'(t) - \gamma_1''(t) \gamma_2'(t)}{\gamma_1'(t)^2} \phi(\gamma_1(t)) dt \\ &=  -\int_d^e \left. \frac{\gamma_2'' \gamma_1' - \gamma_1'' \gamma_2'}{(\gamma_1')^2} \right\vert_{\gamma_1^{-1}(x)} (\gamma_1^{-1})'(x) \phi(x) dx .
\end{align*}
Because derivatives of $\gamma_1$ and $\gamma_1^{-1}$ are $L^\infty_{loc}$, we obtain
\begin{equation*}
 \left( \frac{\gamma_2'' \gamma_1' - \gamma_1'' \gamma_2'}{(\gamma_1')^2} \circ \gamma_1^{-1} \right) \cdot (\gamma_1^{-1})' \in L^2(d,e).
\end{equation*}
The claim follows.
\end{proof}

\begin{proof}[Proof of Proposition \ref{prop:concavity}]
We show that for every given $u \in G_\psi$ there exists a concave $v \in G_\psi$ such that $\mathcal{E}_\epsilon(v) \leq \mathcal{E}_\epsilon(u)$. For this, we adopt the construction in \cite[Lemma 2.1]{Anna}. Here we have to replace only the penalization term. The construction gives us $v \in G_\psi$ such that 
\begin{itemize}
\item $v$ is concave,
\item $\mathcal{E}(v) \leq \mathcal{E}(u)$, 
\item If we define $I:= \{x \in (0,1) : u(x) = v(x) \} $. Then $u' = v' $ on $I$ and $v'' = 0$ on $(0,1) \setminus I$.
\end{itemize}
Now it remains to show that 
\begin{equation*}
\int_0^1 \sqrt{1+ v'^2}  dx \leq \int_0^1 \sqrt{1+ u'^2} dx
\end{equation*}
Using the given properties and the estimate 
\begin{equation*}
\sqrt{1+ u'^2} - \sqrt{1+ v'^2 } \geq \frac{v'}{\sqrt{1+ v'^2}} (u' - v') ,
\end{equation*}
which is due to the convexity of $x \mapsto \sqrt{1+ x^2}$, we find
 proceeding similarly to the rest of the proof of \cite[Lemma 2.1]{Anna}
\begin{eqnarray*}
\int_0^1 \sqrt{1 + v'^2} dx & = & \int_I \sqrt{1 + u'^2 } dx + \int_{(0,1) \setminus I } \sqrt{1+ v'^2  } dx    \\ & \leq & \int_0^1 \sqrt{1+ u'^2} dx - \int_{(0,1) \setminus I } \frac{v'}{\sqrt{1 + v'^2}} ( u' - v') dx.
\end{eqnarray*}

If now $(a,b)$ is a connected component of $(0,1) \setminus I$ then $v'$ is constant on $(a,b)$, say $v' \equiv C$. Therefore  
\begin{eqnarray*}
\int_{(a,b)} \frac{v'}{\sqrt{1+ v'^2}}(u' - v') & = & \frac{C}{\sqrt{1+ C^2}} \int_a^b (u'(x) - v'(x)) dx  \\ & = & \frac{C}{\sqrt{1+ C^2}} (u(a) - v(a)- u(b) + v(b) ) = 0 ,
\end{eqnarray*}
since $a,b \in I$. Therefore 
\begin{equation*}
\int_{(0,1) \setminus I } \frac{v'}{\sqrt{1 + v'^2}} ( u' - v') dx = 0 ,
\end{equation*}
which completes the proof.
\end{proof}

\begin{proof}[Proof of Proposition \ref{prop:bpunkt}] 
We follow the lines of \cite[Corollary 3.4]{Anna}. Assume that $u> \psi$ on $(0,1)$. Then $u \in C^\infty([0,1])$. Proposition \ref{prop:p9} (2) yields that 
\begin{equation*}
\min_{x \in [0,1] } v(x) = \min(v(0), v(1)) = 0
\end{equation*}
and hence $v \geq 0 $. On the other hand, concavity implies that $v \leq 0 $. So $v \equiv 0$ and therefore $u'' \equiv 0$. Since $u(0) = u(1) = 0$ we find that $u \equiv 0$ which is impossible since we required the obstacle to be strictly positive at at least one point. 
\end{proof}
\section{Hypergeometric Functions}

\begin{definition}(Hypergeometric Function, \cite[Definition 2.1.5]{Andrews})\label{def:hypgeo}

Let $a,b,c,z \in \mathbb{C}$. Define for $n \in \mathbb{N}$
\begin{equation*}
(a)_n := \frac{\Gamma(a+n)}{\Gamma(n)}.
\end{equation*}
Then we define $\mathit{HYP2F1}$ to be the analytic continuation of the holomorphic function 
\begin{equation*}
\{z \in \mathbb{C} | \; |z| < 1 \} \ni w \mapsto \sum_{ w = 0 }^\infty \frac{(a)_n (b)_n}{(c)_n} \frac{w^n}{n!} \in \mathbb{C},
\end{equation*}
whenever the series converges conditionally.
\end{definition}

\begin{prop}\label{prop:pfaff} (Pfaff's Transformation for Hypergeometric Functions,\cite[Theorem 2.2.5]{Andrews})

For each $a,b,c,z \in \mathbb{C}$ it holds that 
\begin{equation*}
\mathit{HYP2F1} ( a,b,c, z) = \frac{1}{(1-z)^a} \mathit{HYP2F1}  \left(a, c-b , c , \frac{z}{z-1} \right)  , 
\end{equation*}
as equality of meromorphic functions.
\end{prop}

\begin{prop} \label{prop:valat1} (Values of the Hypergeometric Functions at $z = 1$, \cite[Theorem 2.2.2]{Andrews})

Assume that $\mathrm{Re}(c-b-a ) > 0 $. Then
\begin{equation*}
\mathit{HYP2F1} (a,b,c, 1) = \frac{\Gamma(c) \Gamma(c-b-a) }{\Gamma(c-a) \Gamma(c-b) }.
\end{equation*}
\end{prop}

\begin{lemma}\label{lem.induction}
Let $\theta \in \mathbb{N}_0$. Then 
\begin{equation*}
\int_0^1 \frac{s^\theta}{\sqrt{1-s}} = \prod_{l=1}^\theta \frac{2 l}{(2l+1) }.
\end{equation*}
\end{lemma}
\begin{proof}
We proceed by induction over $\theta$. The case $\theta = 0 $ is very obvious. Now suppose the claim is true for some $\theta \in \mathbb{N}_0$. 
\begin{align*}
\int_0^1 \frac{s^{\theta + 1} }{\sqrt{1-s}} ds & = \left[ -  2 \sqrt{1-s} s^{\theta + 1} \right]^{1}_{0} + 2(\theta + 1)\int_0^1 \sqrt{1-s} s^\theta ds
\\ & = (2 \theta + 2)\left( \int_0^1 \frac{s^\theta}{\sqrt{1-s}}
 -\int_0^1 \frac{s^{\theta+ 1}}{\sqrt{1-s}} \right)    
\end{align*} 
and therefore 
\begin{equation*}
\int_0^1 \frac{s^{\theta+1}}{\sqrt{1-s}} ds = \frac{2 \theta+ 2 }{2 \theta + 3} \int_0^1 \frac{s^\theta}{\sqrt{1-s}}ds = \prod_{ l = 1}^\theta \frac{2l }{2l+ 1 }. \qedhere
\end{equation*}
\end{proof}

\begin{lemma} \label{lem:intiden}(Integral identities for hypergeometric functions)

For each $A > 0 $ we have 
\begin{equation*}
\int_0^A \frac{1}{\sqrt{A-t}(1+t^2)^\frac{5}{4}} = \sqrt{A} \cdot \mathit{HYP2F1} \left(1, \frac{1}{2}, \frac{3}{4}, - A^2 \right) ,
\end{equation*}
\begin{equation*}
\int_0^A \frac{t}{\sqrt{A-t}(1+t^2)^\frac{5}{4}} = \frac{2}{3} A^\frac{3}{2} \mathit{HYP2F1} \left( 1, \frac{3}{2} , \frac{7}{4}, -A^2 \right) .
\end{equation*}
\end{lemma}
\begin{proof}
It is sufficient to show the claim for $A \in (0,1)$ since both expressions are analytic on $\{\mathrm{Re}(A) > 0 \} $ as $z \mapsto \sqrt{z}$ is holomorphic on $\{ \mathrm{Re}(z) > 0 \} $.  Using Lemma \ref{lem.induction} and $(-1)^k{\alpha \choose k}= \frac{1}{k!}(\alpha)_k$ for each $k \in \mathbb{N}_0$ and $\alpha \in \mathbb{R}$ we find
\begin{align*}
\int_0^A \frac{1}{\sqrt{A-t}} \frac{1}{(1+t^2)^\frac{5}{4}} dt & = \sqrt{A} \int_0^1 \frac{1}{\sqrt{1-s}} \frac{1}{(1+ A^2 s^2)^\frac{5}{4}} ds 
\\ &  = \sqrt{A} \int_0^1 \frac{1}{\sqrt{1-s}} \sum_{k = 0}^\infty (-1)^k  {{- \nicefrac{5}{4} } \choose {k}} s^{2k} ( -A^2)^k 
\\ & = \sqrt{A}  \sum_{k= 0 }^\infty (-1)^k {- \nicefrac{5}{4} \choose k } (-A^2)^k \int_0^1 \frac{s^{2k}}{\sqrt{1-s}} ds 
\\ & = \sqrt{A}\sum_{k= 0 }^\infty (-1)^k {- \nicefrac{5}{4} \choose k } (-A^2)^k \prod_{l = 1}^{2k} \frac{2l}{2l+1}
\\ & = \sqrt{A} \sum_{k = 0}^\infty \frac{(\nicefrac{5}{4})_k}{k!}\prod_{l = 1}^{2k} \frac{2l}{2l+1} (- A^2)^k 
\\ & = \sqrt{A} \sum_{k = 0}^\infty \frac{5\cdot 9 \cdot ...\cdot (1+4k)}{4^k k! } \frac{2 \cdot 4 \cdot ... \cdot (4k)}{3 \cdot 5 \cdot ... \cdot (4k+1)} (-A^2)^k 
\\ & = \sqrt{A} \sum_{k = 0 }^\infty \frac{2 \cdot 4 \cdot ... \cdot (4k) }{4^k k! \left[ \prod_{l = 0}^{k-1} (3+4k) \right] } (-A^2)^k
 = \sqrt{A} \sum_{k = 0}^\infty  \frac{2 \cdot 4 \cdot ... \cdot (4k) }{k! 16^k  ( \nicefrac{3}{4})_k } (-A^2)^k 
\\ & = \sqrt{A} \sum_{k = 0}^\infty \frac{2 \cdot 4 \cdot ... \cdot (4k) }{2^{2k}\cdot 2^{2k} \cdot  ( \nicefrac{3}{4} )_k k! } (-A^2)^k 
 = \sqrt{A} \sum_{k = 0}^\infty \frac{(2k)!}{2^{2k} k! (\nicefrac{3}{4})_k } (-A^2)^k 
\\ & = \sqrt{A} \sum_{k = 0}^\infty  \frac{\prod_{l = 0 }^{k-1} (2l + 1)}{2^k} \frac{1}{( \nicefrac{3}{4})_k } (-A^2)^k  = \sqrt{A} \sum_{k = 0}^\infty \prod_{l = 0 }^{k-1} \left( \frac{1}{2} + l \right)  \frac{1}{(\nicefrac{3}{4})_k} (-A^2)^k 
\\ &= \sqrt{A} \sum_{k = 0 }^\infty \frac{(\nicefrac{1}{2})_k (1)_k}{(3/4)_k k!} (-A^2)^k = \sqrt{A} \cdot \mathit{HYP2F1}( 1, \nicefrac{1}{2} , \nicefrac{3}{4} , -A^2) .
 \end{align*}
 The second identity follows using very similar techniques. 
\end{proof}

\begin{lemma} (Completion of the Proof of Theorem \ref{thm:nonexcone})  \label{lem:finiteness}

Define for $A \in (0,\infty)$
\begin{equation*}
G(A):=  \frac{1}{3}A \frac{\mathit{HYP2F1}( 1, \frac{3}{2} ; \frac{7}{4} , -A^2) }{\mathit{HYP2F1}(\frac{1}{2},1, \frac{3}{4}, -A^2)}.
\end{equation*}
Then
\begin{equation*}
\lim_{A \rightarrow \infty} G(A) = \frac{1}{3} \frac{\Gamma( \frac{7}{4}) \Gamma( \frac{1}{4}) }{\Gamma(\frac{3}{4})^2\Gamma(\frac{3}{2})} \simeq 0.834626
\end{equation*}
\end{lemma}
\begin{proof}
 Using Proposition \ref{prop:pfaff} in both numerator and denominator we find 
 \begin{equation*}
 G(A) = \frac{1}{3} \frac{\sqrt{1+ A^2}}{ \mathit{HYP2F1}(1, - \frac{1}{4}, \frac{3}{4}, \frac{A^2}{A^2 + 1} ) } \frac{A}{(1+A^2) } \mathit{HYP2F1}\left( 1 , \frac{1}{4} , \frac{7}{4} , \frac{A^2}{A^2 + 1} \right) 
 \end{equation*}
Cancelling in numerator and denominator we obtain
\begin{equation*}
\lim_{A \rightarrow \infty} G(A) = \lim_{A \rightarrow \infty } \frac{A}{\sqrt{1+A^2}}\frac{1}{3}\frac{\mathit{HYP2F1} ( 1, \frac{1}{4}, \frac{7}{4}, \frac{A^2}{A^2 + 1}  )  }{\mathit{HYP2F1} ( \frac{1}{2} , \frac{-1}{4} , \frac{3}{4}, \frac{A^2}{A^2 + 1}) } = \frac{1}{3} \frac{\Gamma( \frac{7}{4}) \Gamma( \frac{1}{4}) }{\Gamma(\frac{3}{4})^2\Gamma(\frac{3}{2})}, 
\end{equation*} 
where we used Proposition \ref{prop:valat1} for the last identity step. We also would have to use continuity of $\mathit{HYP2F1}$ at 1, which however follows from analyticity.
\end{proof}

\section{Proofs of Results in Section 4.1}
\begin{proof}[Proof of Proposition \ref{prop:p12}]
Since $\gamma_1\in W^{2,1}_{loc}((0,1);\mathbb{R}^2)$, it lies in $C^1((0,1);\mathbb{R}^2)$, so it is locally Lipschitz. Therefore using $\gamma_1' \geq 0 $ and \cite[Theorem 1, Section 3.4.2]{Evansgariepy}, we find
\begin{eqnarray*}
1 = \int_0^1 \gamma_1' dt & =&  \lim_{n \rightarrow \infty} \int_{\frac{1}{n}}^{1- \frac{1}{n}} |\gamma_1' | dt \\ &= & \lim_{n \rightarrow \infty} \int_{\mathbb{R}} \mathcal{H}^0( \{ t \in ( \nicefrac{1}{n} , 1- \nicefrac{1}{n})\;  | \;  \gamma_1(t) = x \} ) dx  .
\end{eqnarray*}
Note that the integrand is actually measurable because of \cite[Lemma 2 (ii), Section 3.3]{Evansgariepy}. The integrand is monotone in $n$ and converges pointwise - as $\mathcal{H}^0$ is a measure - to 
\begin{equation*}
x \mapsto \mathcal{H}^0 \left( \bigcup_{n \in \mathbb{N}} \{ t \in ( \nicefrac{1}{n} , 1- \nicefrac{1}{n})\;  | \;  \gamma_1(t) = x \} \right) = \mathcal{H}^0( \{ t \in ( 0 , 1 )\;  | \;  \gamma_1(t) = x \}). 
\end{equation*}
The monotone convergence theorem implies that 
\begin{equation}\label{eq:dingsda}
1  = \int_{\mathbb{R}} \mathcal{H}^0( \{ t \in ( 0 , 1 )\;  | \;  \gamma_1(t) = x \}) dx . 
\end{equation}
Now $\gamma_1 : [0,1] \rightarrow [0,1]$ is surjective since $\gamma_1$ is continuous, monotone and $\gamma_1(0) = 0$, $\gamma_1(1) = 1$. The intermediate value theorem and the fact that $\mathcal{H}^0$ is the counting measure implies that 
\begin{equation*}
\mathcal{H}^0 ( \{ t \in (0,1) | \gamma_1(t) = x \} )  \geq 1 \quad \forall x \in (0,1).
\end{equation*} 
If $\mathcal{H}^0( \{ t \in (0,1) | \gamma_1(t) = x \} ) >  1$ on a subset of $(0,1)$ of nonzero Lebesgue measure, then \eqref{eq:dingsda} would fail to hold. This proves the claim. 
\end{proof}
\begin{proof}[Proof of Proposition \ref{prop:p15}]
Follows immediately from the fact that $W^{1,1}$ functions are absolutely continuous (see  \cite[Theorem 2.17]{Buttazzogiaquinta}), and have also bounded variation, and the Banach-Zaretsky Theorem, \cite[Theorem 4.6.2]{Benedetto}.
\end{proof}
\begin{proof}[Proof of Proposition \ref{prop:p11}]
Since monotone functions are a.e. differentiable, there is a null set $N_1$ such that $v^{-1}$ is differentiable in a classical sense on $(v(a),v(b)) \setminus N_1$. There is another null set $N_2$ such that $v$ is differentiable on $(a,b) \setminus N_2$. Thanks to Proposition \ref{prop:p15} also $v(N_2)$ is a null set. Now for each $x_0 \in (v(a),v(b))\setminus N_1 \cup v(N_2)$  it holds that 
\begin{eqnarray*}
1 & = & \lim_{x \rightarrow x_0} \frac{x-x_0}{x- x_0} = \lim_{ x \rightarrow x_0} \frac{v(v^{-1}(x))- v (  v^{-1}(x_0))}{v^{-1}(x)- v^{-1}(x_0)} \frac{v^{-1}(x) - v^{-1}(x_0)}{x- x_0} \\ & = & (v^{-1})'(x_0) v'(v^{-1}(x_0)) .
\end{eqnarray*}
The claim follows from solely this equation, pointing out particularly that the last line implies that $v' \circ v^{-1}$ is measurable since it is nonzero almost everywhere and its reciprocal coincides with a measureable function a.e. 
\end{proof}
\begin{proof}[Proof of Proposition \ref{prop:p14} ]
Clearly $u \circ v$ is bijective and $v^{-1} \circ u^{-1} \in W^{1,1} $ follows directly from \cite[Section 4.2.2, Theorem 4]{Evansgariepy}. For $u \circ v \in W^{1,1}$ observe that, for $\phi \in  C_0^\infty(a,b)$ 
\begin{align*}
\int_a^b u(v(x))\phi'(x) dx  &= \int_{v(a)}^{v(b)} u(y) \phi'(v^{-1}(y)) (v^{-1})'(y) dy \\ & = \int_{v(a)}^{v(b)} u(y) (\phi \circ v^{-1})'(y) dy = -\int_{v(a)}^{v(b)} u'(y) \phi(v^{-1}(y) ) dy \\ &= -\int_a^b u'(v(x)) v'(x) \phi(x) dx .
\end{align*}
and $(u' \circ v )v' \in L^1(a,b) $ according to  \cite[Theorem 263 D]{Fremlin}. 
\end{proof}
\begin{proof}[Proof of Proposition \ref{prop:p16}]
The composition $u \circ v$ is continuous on $[a,b]$ so certainly $L^1(a,b)$. Now let $N$ be a set of measure zero such that $v$ is differentiable and $v' \neq 0 $ on $N^C$, see Proposition \ref{prop:p11}. Let $\phi \in C_0^\infty(a,b)$. Then using \cite[Theorem 263 D]{Fremlin}, Proposition \ref{prop:p15} and the product rule in $W^{1,1}(a,b)$ 
\begin{align*}
\int_a^b (u \circ v)(x) \phi'(x) dx  & =  \int_a^b u(v(x)) \phi'(x) dx 
\\ &= \int_{(a,b) \setminus N} u(v(x) ) \phi'(v^{-1}(v(x))) \frac{v'(x)}{v'(x)} dx 
\\ & = \int_{ v((a,b) \setminus N ) } u(z) \phi'(v^{-1}(z)) \frac{1}{v'(v^{-1}(z))} dz   \\ & =\int_{v(a)}^{v(b)} u(z) (\phi \circ v^{-1})'(z) dz = - \int_{v(a)}^{v(b)} u'(z) \phi(v^{-1}(z)) dz \\ 
& = - \int_{v((a,b) \setminus N )} u'(z) \phi(v^{-1}(z)) dz =  - \int_{(a,b) \setminus N } u'(v(x)) v'(x) \phi(x) dx \\ & =  - \int_a^b u'(v(x)) v'(x) \phi(x) dx .
\end{align*} 
Now $(u' \circ v)v' \in L^1$ follows from in \cite[Theorem 263 D]{Fremlin} and this proves the claim.   
\end{proof}

\begin{proof}[Proof of Proposition \ref{prop:p18}]
Let $(\rho_\epsilon)_{\epsilon > 0 } $ be the standard mollifier. Since the Picard-Lindelöf Theorem is not directly applicable we modify the differential equation first. Define for each $\epsilon > 0$ $\phi_\epsilon$ to be the global solution of 
\begin{equation*}
\begin{cases}
\phi_\epsilon'(s) = \frac{1}{\sqrt{1 + (u' * \rho_\epsilon)(\phi_\epsilon(s))^2}}  &  \\ \phi_\epsilon(0) = 0 
\end{cases}
\end{equation*}
where we tacitly extend $u'$ by zero on all of $\mathbb{R}$ to make the convolution well-defined and smooth on $\mathbb{R}$. Now note that for each $\widetilde{L}> 0$ the norm  $||\phi_\epsilon||_{W^{1,2}(0,\widetilde{L})}$ is bounded, since $\phi_\epsilon(0) = 0 $ and $||\phi_\epsilon'||^2_{L^2(0,\widetilde{L})} \leq \widetilde{L} $. Therefore a subsequence, which we denote again by $\phi_\epsilon $, converges to some $\phi_{\widetilde{L} } \in W^{1,2}(0,\widetilde{L})$ weakly. If we consider $L'> \widetilde{L}$ the restriction of $\phi_{L'}$ on $[0, L]$ coincides with $\phi_L$, a.e.. So we will leave out the index from now on and simply write $\phi$. Notice  that $\phi$ is continuous and increasing. Let
\begin{equation*}
E:= \{ x \in (0,\infty) | \;  \phi(x) \in (0,1) \}. 
\end{equation*}
Then, $E$ is an open interval. We will now show that 
\begin{equation*}
(u' * \rho_\epsilon) \circ \phi_\epsilon \rightarrow u' \circ \phi \quad \textrm{pointwise a.e. on} \;  E .
\end{equation*}
First observe that for each $\widetilde{L} > 0$, $\phi_\epsilon$ converges to $\phi$ uniformly in $[0,\widetilde{L}]$ because of compactness of the embedding $W^{1,2} \hookrightarrow C^0$. Now fix $x \in (0,\infty)$ such that $\phi(x) \in (0,1)$ and let $\widetilde{L} > 0 $ such that $x \in [0, \widetilde{L}]$ . Then
\begin{align*}
| ((u' * \rho_\epsilon) \circ \phi_\epsilon)(x) - & (u' \circ \phi)(x) |  \\ & \leq  \int_{B_\epsilon(\phi_\epsilon(x)) } | u'(y) - u'(\phi(x) ) | \rho_\epsilon( \phi_\epsilon(x) - y) dy \\ & \leq  \sup_{y \in B_\epsilon(\phi_\epsilon(x) ) }| u'(y) - u'(\phi(x)) |  \int_{B_\epsilon(\phi_\epsilon(x)) } \rho_\epsilon(\phi_\epsilon(x) -y) dy  \\ & \leq  \sup_{y \in B_\epsilon(\phi_\epsilon(x) )} | u'(y) - u'(\phi(x)) | . 
\end{align*}
If $y \in B_\epsilon(\phi_\epsilon(x)) $ then 
\begin{equation*}
|y - \phi(x) | \leq |y- \phi_\epsilon(x) | + ||\phi_\epsilon - \phi ||_{L^\infty(0, \widetilde{L})} \leq \epsilon + || \phi_\epsilon - \phi||_{L^\infty(0, \widetilde{L}) }  .
\end{equation*}
Since $\phi(x) $ was assumed to be an element of $(0,1)$, there is $\epsilon_0 > 0 $ such that $|\phi(x) | , |1 - \phi(x) |> \epsilon + || \phi_\epsilon - \phi||_{L^\infty(0, \widetilde{L}) }  $ for each $\epsilon < \epsilon_0$. Then 
\begin{equation*}
| ((u' * \rho_\epsilon) \circ \phi_\epsilon)(x) - (u' \circ \phi)(x) |  \leq \sup_{y \in B_{\epsilon + || \phi_\epsilon- \phi||_{L^\infty(0,\widetilde{L})}(\phi(x))} }  | u'(y) - u'(\phi(x) ) | \rightarrow 0 \quad ( \epsilon \rightarrow 0 ),
\end{equation*}
because of continuity of $u'$. 

We now show that $\phi$ solves the desired differential equations and lies in $C^1(0,L) \cap W^{1,1}(0,L)$ for some $L>0$. Using the dominated convergence theorem we find that 
\begin{equation*}
\phi_\epsilon' = \frac{1}{\sqrt{1+[(u'* \rho_\epsilon) \circ \phi_\epsilon]^2}} \rightarrow \frac{1}{\sqrt{1+ (u' \circ \phi)^2}} 
\end{equation*}
in $L^2(0,\widetilde{L})$ for each $\widetilde{L} >0 $ . Therefore $\phi' = \frac{1}{\sqrt{1+ (u' \circ \phi)^2}}$ a.e.  on $[0,\infty)$ which implies that $\phi$ solves the prescribed differential equation in $E$ and thus  $\phi \in C^1(E)$. Now $E$ is a finite interval since, with $\lambda$ denoting the Lebesgue measure, we find that 
\begin{eqnarray*}
\lambda(E) & =  & \int_E 1 ds = \int_E \frac{\sqrt{1+ u'(\phi(s))}}{\sqrt{1+ u'(\phi(s))}} ds = \int_E \phi'(s) \sqrt{1+ u'(\phi(s))} ds  \\ & =  & \int_0^1 \sqrt{1+ u'^2}dx < \infty.  
\end{eqnarray*} 
Therefore $E = (0,L) $ for some $L> 0$ as claimed, which also implies that $\phi \in W^{1,1}(0,L)$. Note from the equation for $\phi'$ can also be inferred that $\phi \in C^1(0,L)$.

To show that $(\phi, u \circ \phi)$ is a (weak) reparametrization of $(x,u(x))$ one needs to show that $\phi$ is invertible in $W^{1,1}$. Since $\phi \in C^1(0,L)$ and $\phi' >0 $ on $(0,L)$, $\phi$ is bijective and possesses an inverse $\phi^{-1}$ such that for each $x \in (\phi(0), \phi(L)) = (0,1)$ it holds that
\begin{equation*}
(\phi^{-1})'(x) = \frac{1}{\phi'(\phi^{-1}(x))} = \sqrt{1+ u'(x)^2},
\end{equation*}
which has finite integral due to the fact that $u \in W^{1,1}(0,1)$. This implies the claim. It remains to show that if $u \in W^{2,2}_{loc}(0,1)$, then $\phi\in W^{2,2}_{loc}(0,L)$ and the reparametrization lives in $W^{2,1}(0,L)$ provided that $u$ is concave. For the first assertion, fix $\psi \in C_0^\infty (0,L )$ and note that $\psi \circ \phi^{-1} \in C_0^1(0,1)$ and 
\begin{align}\label{eq:weakdif}
 & \int_0^L \phi'(s) \psi'(s) ds =  \int_0^L \frac{1}{\sqrt{1+u'(\phi(s))^2}} \psi'(s) ds \\ & \qquad =  \int_0^1 \frac{1}{\sqrt{1+ u'(x)^2}} (\psi \circ \phi^{-1})'(x) dx =  -\int_0^1 \frac{-u''(x) u'(x) }{(1+ u'(x)^2)^\frac{3}{2}} \psi(\phi^{-1}(x)) dx \nonumber \\ & \qquad = - \int_0^L \frac{- u''(\phi(s)) u'(\phi(s))}{(1+ u'(\phi(s))^2)^2} \psi(s) ds . \nonumber
\end{align} 
Clearly, this weak derivative lies in $L^2_{loc}(0,L)$ and hence $\phi \in W^{2,2}_{loc} (0,L)$. If $u$ is additionally concave then $u'' \leq 0 $ a.e. yields that 
\begin{equation*}
\int_0^L |\phi''(s)| ds = \int_0^L \frac{- u''(\phi(s))| u'(\phi(s))|}{(1+ u'(\phi(s))^2)^2} = \int_0^1 \frac{ - u''(x)| u'(x)|}{(1 + u'(x)^2 )^\frac{3}{2}} dx .
\end{equation*}
Using the monotone convergence theorem we find 
\begin{equation*}
\int_0^L |\phi''(s)| ds = \lim_{\epsilon \rightarrow 0 } \int_{\epsilon}^{1-\epsilon}  \frac{ - u''(x) |u'(x)|}{(1 + u'(x)^2 )^\frac{3}{2}} dx = \lim_{\epsilon \rightarrow 0 }  \left( H(u'(1-\epsilon)) - H(u'(\epsilon)) \right). 
\end{equation*}
where 
\begin{equation*}
H(z) = \int_0^z \frac{|y|}{(1+ y^2)^\frac{3}{2}} dy.
\end{equation*}
This limit exists in $\mathbb{R}$ thanks to boundedness and monotonicity of $H \circ u'$. Therefore $\phi \in W^{2,1}$. It remains to show that $u \circ \phi \in W^{2,1}$. With arguments similar to  \eqref{eq:weakdif} one can show that $(u \circ \phi)'$ is weakly differentiable. Now observe that for $s \in (0,L) $
\begin{eqnarray*}
(u \circ \phi)''(s) & =  & u''(\phi(s)) \phi'(s)^2 + u'(\phi(s)) \phi''(s) \\ & =& \frac{u''( \phi(s) )}{1 + u'(\phi(s))^2} - \frac{u'(\phi(s))^2 u''(\phi(s))}{(1+ u'(\phi(s))^2)^2} \\ & =&\frac{u''(\phi(s))}{(1 + u'(\phi(s))^2)^2}.
\end{eqnarray*}
Therefore, again due to concavity and the monotone convergence theorem 
\begin{eqnarray*}
\int_0^L |(  u\circ \phi ) ''(s) | ds & =&  \int_0^L \frac{-u''(\phi(s))}{(1 + u'(\phi(s))^2  )^2} ds = \int_0^1 \frac{-u''(x)}{(1+ u'(x)^2)^\frac{3}{2}} dx \\ &= & \lim_{\epsilon \rightarrow 0 } \left( \frac{u'(\epsilon)}{\sqrt{1+ u'(\epsilon)^2}} - \frac{u'(1- \epsilon)}{\sqrt{1+ u'(1- \epsilon)^2}} \right), 
\end{eqnarray*}
which again exists due to concavity and monotonicity of $z \mapsto \frac{z}{\sqrt{1+z^2}}$. 
\end{proof}

\section{Proofs of Results in Section 4.2}
\begin{proof}[Proof of Lemma \ref{lem:l1}]
It is clear by definition that $u$ is weakly differentiable. Suppose that $K \subset (0,1)$ is compact. Let $( \psi_\epsilon)_{\epsilon > 0 }$ be a sequence of standard mollifiers. Then there is $\epsilon_0 > 0 $ such that $\phi_\epsilon := (\mathrm{sgn}(u') \chi_K )* \psi_\epsilon $ lies in $C_0^\infty(0,1)$ for all $\epsilon < \epsilon_0$. Convolution implies that $||\phi_\epsilon||_\infty \leq 1 $. Now $\phi_\epsilon \rightarrow \mathrm{sgn}(u') \chi_K$ in $L^q(0,1)$,  for $q \in [1, \infty) $ chosen  such that $\frac{1}{p}+ \frac{1}{q} =1$. Hence
\begin{equation*}
\int_K  |u'| dx = \lim_{\epsilon \rightarrow 0 } \int_0^1 u' \phi_\epsilon dx  = \lim_{\epsilon \rightarrow 0 } - \int_0^1 u \phi_\epsilon' dx \leq \int |Du| ,
\end{equation*}
where $\int |Du|$ denotes the total variation of $u$ on $(0,1)$. Now we can exhaust $(0,1)$ by compact sets and the monotone convergence theorem implies that 
$ \int_0^1 |u'| dx < \infty . $
\end{proof} 


\begin{proof}[Proof of Lemma \ref{lem:rgulang}]
For the variational inequality: Since $M_1$ is convex we find for each $v \in M_1$ :
\begin{equation*}
0 \leq \frac{d}{dt}_{\mid_{t=0}} \int_{a}^{b} \sqrt{1 + (u' + t(v'-u'))^2} = \int_{a}^{b} \frac{u'}{\sqrt{1+ u'^2}} (v' - u') dx .
\end{equation*}
Uniqueness of the solution of this inequality can be shown using very elementary arguments. From the variational inequality can be inferred that $u \in C^1([a,b])$, similar to \cite[Section II.7]{Kinderlehrerstampacchia}.

 Now suppose that there is an $x \in (a,b)$ such that $|u'(x)| > ||\psi'||_{\infty, (a,b)}$. This already implies that $u(x) > \psi(x):$ Indeed, if $u(x) = \psi(x) $ then $u, \psi \in C^1 $ implies that $u'(x) = \psi'(x)$, which is a contradiction.
 Therefore $x$ lies in a connected component of $\{y \in (a,b) | u(y) >\psi (y) \}   $. If $D$ is such a connected component, then for each $\psi \in C_0^\infty(D)$ there is $\epsilon>0$ such that for $|t| < \epsilon$, we have that $u + t\psi \in M_1$. Choosing $v = u+t \psi $ in  \eqref{eq:laengenvarungl} for $|t|< \epsilon$ we find that
\begin{equation}\label{eq:konstabl}
\int_a^b \frac{u'}{\sqrt{1+ u'^2}} \psi' = 0 .
\end{equation}
This implies that $u'$ is constant on $D$, but in case that $u $ touches $\psi $ somewhere, there is some boundary point of $D$ where $u$ touches $\psi $ and by virtue of that there is some $z \in (a,b)$ which is a boundary point of $D$ such that $|u'(x)| = |\psi' (z) | \leq || \psi'||_{\infty}  $. A contradiction. In case that $u$ does not touch  $ \psi $ in $(a,b)$, then $ D = (a,b)$ and \eqref{eq:konstabl} yields that $u$ is a line, in which case $|u'(x)| = \frac{|d_2 - d_1|}{b - a}$. We have shown
\begin{equation}\label{eq:defC}
|u'(x)| \leq \max \left( \frac{|d_2 - d_1|}{b - a} , ||\psi'||_\infty \right)  := C.
\end{equation}

 From now on we proceed similar to  \cite[Section 5.4]{Buttazzogiaquinta} to show $W^{2,2}_{loc} $-regularity. For this we first introduce the difference quotient operator. Fix $\delta > 0 $.  For a function $v : (a, b) \rightarrow \mathbb{R}$ and $0 < |h|< \delta $ we set
 \begin{equation*}
 \Delta_h v (x) = \frac{v(x+h) - v(x)}{h} , \quad  x \in (a+ \delta, b-  \delta) .
 \end{equation*}
 Now for each $\eta \in C_0^\infty( a+ 2\delta , b - 2\delta ) $ such that $\eta \geq 0 $  
%
 and $|h|< \delta$ fixed, we set 
  \begin{equation*}
 w_\epsilon(x) := u(x) + \begin{cases}  \epsilon \Delta_{-h} ( \eta^2 \Delta_h (u- \psi)) (x) &  x \in (a+\delta,b- \delta), \\ 0 & \textrm{otherwise} \end{cases},
 \end{equation*}
By  \cite[p.191 below (5.57)]{Buttazzogiaquinta}, $w_\epsilon$ is an element of $M_1$ for sufficiently small $\epsilon > 0 $.
Choose $\eta \in C_0^\infty((a + 2 \delta, b- 2 \delta) )$  such that $0 \leq \eta \leq 1 $ and $ \eta \equiv 1 $ on $(a+ 3\delta, b- 3 \delta)$.  Plugging into \eqref{eq:laengenvarungl} we obtain that  
\begin{equation*}
0 \leq \epsilon \int_a^b \frac{u'}{\sqrt{1+u'^2}} \left( \Delta_{-h} ( \eta^2 \Delta_h (u' - \psi') + 2 \eta \eta' \Delta_h (u-\psi) \right) dx .
\end{equation*}
Dividing by $\epsilon$ and proceeding similar to \cite[Bottom of p. 191]{Buttazzogiaquinta} we obtain that 
\begin{equation*}
\int_a^b \Delta_h \left( \frac{u'}{\sqrt{1+u'^2}} \right) \eta^2 \left(  \Delta_h (u' - \psi') + 2\eta \eta' \Delta_h (u- \psi) \right)  dx \leq 0  
\end{equation*}
and therefore 
\begin{align}\label{eq:diffiquooti}
\int_a^b \eta^2 \Delta_h \frac{u'}{\sqrt{1+u'^2}} \Delta_h u' dx  & \leq \int_a^b \eta^2 \left\vert \Delta_h \frac{u'}{\sqrt{1+u'^2}} \right\vert \; | \Delta_h \psi'| dx  \\ & \quad + 2 \int_a^b |\eta| \; |\eta'| \;  |\Delta_h(u-\psi)| \; \left\vert \Delta_h \frac{u'}{\sqrt{1+u'^2}} \right\vert  dx .
\end{align}
Since $z \mapsto \frac{z}{\sqrt{1+z'^2}}$ is monotone we find that 
\begin{equation*}
\Delta_h \frac{u'}{\sqrt{1+u'^2}} \Delta_h u' = \left\vert \Delta_h \frac{u'}{\sqrt{1+u'^2}} \right\vert \; \left\vert \Delta_h u' \right\vert,
\end{equation*}
in particular the left hand side of \eqref{eq:diffiquooti} is positive. We claim that 
\begin{equation}\label{eq:nonoton}
  \frac{1}{(1+C^2)^\frac{3}{2}} |\Delta_h u'| \leq \left\vert \Delta_h \frac{u'}{\sqrt{1+u'^2}}  \right\vert \leq  | \Delta_h u'| ,
\end{equation}
where $C$ is defined in \eqref{eq:defC}. Indeed, 
\begin{equation*}
\left\vert \Delta_h \frac{u'}{\sqrt{1+u'^2}}  \right\vert = \frac{1}{|h|} \left\vert \int_{u'(x)}^{u'(x+h)} \frac{1}{(1+s^2)^\frac{3}{2}} \right\vert, 
\end{equation*}
whereupon the upper estimate follows immediately and the lower estimate follows from the fact that $||u'||_\infty \leq C$, see \eqref{eq:defC}.
Using the estimate and \eqref{eq:nonoton} in \eqref{eq:diffiquooti} we find 
\begin{equation*}
\frac{1}{(1+C^2)^\frac{3}{2}}\int_a^b \eta^2 |\Delta_hu'|^2 dx \leq \int_a^b \eta^2 |\Delta_h u'| \; | \Delta_h \psi' |  dx+ \int_a^b |\eta \Delta_h u'|  \; |\eta' \Delta_h(u - \psi) | dx. 
\end{equation*}
Using the Peter-Paul inequality we find for arbitrary $\epsilon > 0$:
\begin{equation}\label{eq:peterpauil}
\frac{1}{(1+C^2)^\frac{3}{2}}\int_a^b \eta^2 |\Delta_hu'|^2 dx \leq 2 \epsilon \int_a^b \eta^2 |\Delta_h u'|^2 dx + \frac{1}{4\epsilon} \int_a^b \eta^2 (\Delta_h \psi')^2 dx + \frac{1}{4\epsilon} \int_a^b \eta'^2 ( \Delta_h (u- \psi) ) ^2 dx.
\end{equation}
Observe that by \cite[Section 5.8.2, Theorem 3 (i)]{Evans} there is $D = D(\delta) > 0$ such that for sufficiently small $h > 0 $ 
\begin{equation*}
\int_{\mathrm{supp}(\eta) } (\Delta_h \psi')^2 dx \leq D \int_{B_\delta(\mathrm{supp}(\eta))} \psi''^2 dx \leq D||\psi''||_{L^2(a+\delta, b- \delta)}^2 
\end{equation*}
and by Lipschitz continuity of $\psi$ and  $u$ we have 
\begin{equation*}
|\Delta_h(u-\psi) | \leq ||u'- \psi'||_\infty \leq C  + ||\psi'||_\infty.
\end{equation*}
Choosing $\epsilon = \frac{1}{4(1+C^2)^\frac{3}{2}}$ we obtain 
\begin{equation*}
\int_{a + 3 \delta}^{b - 3\delta} (\Delta_h u')^2 dx \leq  \int_a^b \eta^2 |\Delta_h u'|^2 dx \leq \frac{(1+C^2)^3}{2} \left( D||\psi''||_{L^2(a+ \delta, b- \delta)} + ||\eta'||_\infty (C+ ||\psi'||_\infty) (b-a) \right).
\end{equation*}
From \cite[Section 5.8.2, Theorem 3 (ii)]{Evans} we conclude that $u'' \in L^2_{loc}(a,b)$ and the claim follows. 
\end{proof}

\subsection*{Acknowledgments}
The author would like to thank Anna Dall'Acqua and the anonymous referee for very helpful suggestions and comments.


\begin{thebibliography}{SK}




\normalsize
\baselineskip=17pt






\bibitem{Ambrosio}
\textsc{Ambrosio, L., Casalles, V., Masnou, S., Morel, J.}
\emph{Connected components of sets of finite perimeter and applications to image processing.}
J. Eur. Math. Soc. (JEMS) 3 (2001), no. 1, 39–92.

\bibitem{Andrews}
\textsc{Andrews, G., Askey, R., Roy, R.}
\emph{Special Functions.}
Cambridge University Press, Cambridge (1999). 

\bibitem{Benedetto}
\textsc{Benedetto, J., Czaja, W. }
\emph{Integration and Modern Analysis.}
Birkh\"auser Boston, Inc., Boston, MA (2009).

\bibitem{Braides}
\textsc{Braides, A.} 
\emph{$\Gamma$-convergence for Beginners.}
Oxford University Press, Oxford (2002). 

\bibitem{Brezis}
\textsc{Brezis, H.}
\emph{Functional Analysis, Sobolev Spaces and Partial Differential Equations.}
Universitext. Springer, New York, 2011.

\bibitem{Buttazzogiaquinta}
\textsc{Buttazzo, G., Giaquinta, M., Hildebrandt, S.}
\emph{One-dimensional Variational Problems.}
The Clarendon Press, Oxford University Press, New York (1998)

\bibitem{Caffarelli}
\textsc{Caffarelli, L., Friedman, A.}
\emph{The Obstacle Problem for the Biharmonic Operator.}
Ann. Scuola Norm. Sup. Pisa Cl. Sci. (4) 6 (1979), no. 1, 151–184.

\bibitem{Anna}
\textsc{Dall'Acqua, A., Deckelnick, K.}
\emph{An Obstacle Problem for Elastic Graphs.}
SIAM J. Math. Anal. \textbf{50} (2018) 119-137. 



\bibitem{Dayrens}
\textsc{Dayrens, F., Masnou, S., Novaga, M,}
\emph{Existence, Regularity and Structure of Confined Elasticae.}
ESAIM Control Optim. Calc. Var. 24 (2018)   


\bibitem{Deckelnick}
\textsc{Deckelnick, K., Grunau, H.}
\emph{Boundary Value Problems for the One-Dimensional Willmore Equation.}
Calc. Var. Partial Differential Equations, \textbf{30} (2007) 293-314. 

\bibitem{Evans}
\textsc{Evans, L.}
\emph{Partial Differential Equations.} 
American Mathematical Society, Providence, RI (2010).

\bibitem{Evansgariepy}
\textsc{Evans, L., Gariepy, R.}
\emph{Measure Theory and Fine Properties of Functions.}
CRC Press Boca Raton, FL (2015). 

\bibitem{Ferriero}
\textsc{Ferriero, A., Fusco, N.}
\emph{A Note on the Convex Hull of Sets of Finite Perimeter in the Plane.}
Discrete Contin. Dyn. Syst. Ser. B 11 (2009), no. 1, 102–108. 

\bibitem{Fremlin}
\textsc{Fremlin, D.}
\emph{Measure Theory. Vol. 2.}
Torres Fremlin, Colchester (2003).

\bibitem{Giusti}
\textsc{Giusti, E.}
\emph{Minimal Surfaces and Functions of Bounded Variation.}
Birkhäuser Verlag, Basel (1984).


\bibitem{Kinderlehrerstampacchia}
\textsc{Kinderlehrer, D., Stampacchia, G.} 
\emph{An Introduction to Variational Inequalities and their Applications.}
Society for Industrial and Applied Mathematics (SIAM), Philadelphia, PA (2000).

\bibitem{Kirchheim}
\textsc{Kirchheim, B., Kristensen, J.} 
\emph{Differentiability of Convex Envelopes.}
C. R. Acad. Sci. Paris S\'er. I Math. \textbf{333} (2001) 725-728.


\bibitem{Mandel}
\textsc{Mandel, R.}
\emph{Boundary Value Problems for Willmore Curves in $\Bbb{R}^2$.}
Calc. Var. Partial Differential Equations \textbf{54} (2015) 3905-3925.



\bibitem{Miura2}
\textsc{Miura, T.}
\emph{Singular Perturbation by Bending for an Adhesive Obstacle Problem.}
Calc. Var. Partial Differential Equations \textbf{55} (2016) Art. 19, 24.  


\bibitem{Miura3}
\textsc{Miura, T.}
\emph{Elastic Curves and Phase Transitions.}
Preprint (2017).

\bibitem{Okabe}
\textsc{Novaga, M., Okabe, S.}
\emph{Regularity of the Obstacle Problem for the Parabolic Biharmonic Equation.}
Math. Ann. 363 (2015), no. 3-4, 1147–1186.

\bibitem{Oberman}
\textsc{Oberman, A.}
\emph{The Convex Envelope is the Solution of a Nonlinear Obstacle Problem.}
Proc. Amer. Math. Soc. \textbf{135} (2007) 1689-1694. 

\bibitem{Toponogov}
\textsc{Toponogov, V.}
\emph{Differential geometry of curves and surfaces.}
Birkhäuser Boston, Inc., Boston, MA, 2006.

%
%





\end{thebibliography}
\end{document}